\documentclass[11pt,reqno]{amsart}

\usepackage[T1]{fontenc}
\usepackage{lmodern}
\usepackage{microtype}
\usepackage{mathtools}
\usepackage{amssymb}
\usepackage{mathrsfs}
\usepackage{enumitem}
\usepackage{xcolor}
\usepackage[font=small,labelfont=sc,labelsep=period,
  justification=justified,singlelinecheck=false]{caption}
\usepackage{tikz}
\usetikzlibrary{arrows.meta,backgrounds,calc,decorations.pathreplacing}
\usepackage[colorlinks=true,linkcolor=blue!45!black,citecolor=blue!45!black,
  urlcolor=blue!45!black]{hyperref}

\setlist[enumerate]{itemsep=3pt,topsep=5pt,parsep=0pt,partopsep=0pt}
\newtheorem{theorem}{Theorem}[section]
\newtheorem{proposition}[theorem]{Proposition}
\newtheorem{lemma}[theorem]{Lemma}
\newtheorem{corollary}[theorem]{Corollary}
\newtheorem{conjecture}[theorem]{Conjecture}
\theoremstyle{remark}
\newtheorem{remark}[theorem]{Remark}

\newcommand{\C}{\mathbb{C}}
\newcommand{\R}{\mathbb{R}}
\newcommand{\Hh}{\mathbb{H}}
\newcommand{\ii}{\mathrm{i}}
\newcommand{\dd}{\mathrm{d}}
\newcommand{\Disc}{\operatorname{disc}}
\newcommand{\Rea}{\operatorname{Re}}
\newcommand{\Ima}{\operatorname{Im}}
\newcommand{\CBF}{\mathcal{CBF}}

\hypersetup{
  pdftitle={Complete Bernstein Functions and Scaled Ultraspherical Zeros},
  pdfauthor={K. Castillo},
  pdfsubject={Complete monotonicity and Pick-function methods for
  ultraspherical zeros},
  pdfkeywords={Gegenbauer polynomials, ultraspherical zeros,
  complete Bernstein functions, Pick functions, complete monotonicity}
}

\title[Bernstein functions and ultraspherical zeros]
{Complete Bernstein functions and scaled ultraspherical zeros}

\author{K. Castillo}
\address{CMUC, Department of Mathematics, University of Coimbra,
3000-143 Coimbra, Portugal}
\email{math@keniercastillo.com}


\begin{document}

\begin{abstract}
Let $z_{n,j}(\lambda)$ denote the positive zeros, in decreasing
order, of the ultraspherical polynomial $C_n^\lambda$,
$\lambda>-1/2$, with the reduced limiting interpretation at
$\lambda=0$ specified below. Our principal result settles three
higher-monotonicity questions of Gautschi: two as printed and the
natural open-interval form of the third, whose printed endpoint
$\lambda=0$ is singular. For every $n\geq3$,
$$
  \sqrt{\lambda+1}\,z_{n,j}(\lambda)
$$
becomes a complete Bernstein function after translation of its
parameter interval to $(0,\infty)$. For every $n\geq2$,
$$
  \sqrt{\lambda}\,z_{n,j}(\lambda)
$$
is a complete Bernstein function on $(0,\infty)$. For every
$n\geq4$, the largest-zero trajectory
$$
 \sqrt{\lambda+\frac{2n^2+1}{4n+2}}\,z_{n,1}(\lambda)
$$
has the same property, whilst for every other positive zero the
derivative of this scaling fails to be completely monotone.
Separately, an exact calculation in degree $4$ provides a
counterexample to a fourth conjecture of Gautschi, concerning the
linear scaling $\lambda z_{n,j}(\lambda)$.
\end{abstract}

\maketitle

\section{Introduction}

We use the classical Gegenbauer normalisation, characterised by the
generating function
$$
 (1-2xt+t^2)^{-\lambda}
 =\sum_{n=0}^{\infty}C_n^\lambda(x)t^n,
 \quad x\in[-1,1],\quad |t|<1.
$$
For arbitrary complex $x$, the identity is understood as a formal
power series in $t$. Thus $C_n^\lambda$ is the coefficient of $t^n$ in the left-hand
side; see \cite[Eq.~(4.7.23), p.~82]{Szego}. In this normalisation,
we use the rising factorial
$$
 (a)_0:=1,\quad
 (a)_r:=a(a+1)\cdots(a+r-1),\quad r=1,2,\ldots.
$$
Then
$$
 C_n^\lambda(1)=\frac{(2\lambda)_n}{n!}.
$$
There is a normalisation point which must be dealt with at the outset.
For $n\geq1$, the polynomial $C_n^\lambda$ contains the common
factor $\lambda$, and hence $C_n^0$ is identically zero. We define
its reduced value at $\lambda=0$ by
$$
 \mathcal P_n^0(x):=
 \lim_{\lambda\to0}\frac{C_n^\lambda(x)}{\lambda}.
$$
It is the nonzero multiple $2T_n(x)/n$ of the Chebyshev polynomial
characterised by $T_n(\cos\theta)=\cos(n\theta)$; the coefficient
calculation is given in Section~\ref{sec:h}. Throughout the paper, a
zero at $\lambda=0$ means a zero of $\mathcal P_n^0$. This
convention is forced by the generating-function normalisation and
agrees with continuation from $\lambda\neq0$.

For $\lambda>-1/2$, $\lambda\neq0$, the Gegenbauer polynomials
form the symmetric Jacobi family and are orthogonal on $(-1,1)$
with respect to the weight $(1-x^2)^{\lambda-1/2}$. Together with
the preceding reduced interpretation at $\lambda=0$, their zeros
are simple and lie in $(-1,1)$; see
\cite[Theorem~3.3.1, pp.~44--46]{Szego}. Fourier--Gegenbauer expansions make
them a standard tool in approximation theory; see
\cite[Chapter~IV, pp.~58--99, and Chapter~IX,
pp.~244--273]{Szego}.

The positive zeros will be denoted in decreasing order by
$$
 z_{n,1}(\lambda)>\cdots>
 z_{n,\lfloor n/2\rfloor}(\lambda)>0.
$$
Here $\lfloor u\rfloor$ and $\lceil u\rceil$ denote,
respectively, the greatest integer not exceeding $u$ and the least
integer not smaller than $u$.

The variation of ultraspherical zeros with the parameter has a long
history. Under its regularity hypotheses, Markov's classical theorem
gives monotonicity of the zeros when
$\partial_\lambda\log w_\lambda(x)$ is monotone as a function of
$x$; it does not directly answer the scaled questions considered here.
See \cite[Theorem~6.12.1, p.~115]{Szego} and \cite{Ismail1987}. During the 1980s
several authors therefore looked for a positive factor $a(\lambda)$
such that
$$
 a(\lambda)z_{n,j}(\lambda)
$$
would increase even though $z_{n,j}(\lambda)$ itself decreases.
Elbert and Laforgia obtained early monotonicity results in
\cite{ElbertLaforgia1986}. Ahmed, Muldoon, and Spigler then proved in
\cite[Theorem~2]{Ahmed1986} that the factor
$$
 \left(\lambda+\frac{2n^2+1}{4n+2}\right)^{1/2}
$$
works for $-1/2<\lambda\leq3/2$, but did not yet obtain the full
interval $\lambda>-1/2$.

The problem was sharpened in the Problem Section of the 1988 Segovia
proceedings edited by Alfaro, Dehesa, Marcellán, Rubio de Francia,
and Vinuesa: Ismail and Letessier asked whether
$\sqrt{\lambda}\,z_{n,j}(\lambda)$ is increasing for
$\lambda\geq0$; see \cite[pp.~329--330]{IsmailLetessier}. Askey
subsequently proposed the parameter-independent enlargement
$\sqrt{\lambda+1}$. The resulting statement became known as the
Ismail--Letessier--Askey, or ILA, conjecture and attracted a sequence
of partial results during the 1990s; see Gautschi's historical account
\cite[pp.~253--255]{Gautschi2018}. The problem was finally settled in
1999 by Elbert and Siafarikas \cite{ElbertSiafarikas}. Their theorem
was stronger than the original ILA statement: it established the
scaled monotonicity with the $n$-dependent factor above on the full
interval $\lambda>-1/2$, from which the ILA conjecture follows.

Thus, by the end of the 1990s it was known that
$$
 \lambda\longmapsto\sqrt{\lambda+1}\,z_{n,j}(\lambda)
$$
is increasing for $n\geq3$, despite the fact that the unscaled zero
is decreasing. The stronger theorem of Elbert and Siafarikas concerns
the scale
$$
d_n:=\frac{2n^2+1}{4n+2}.
$$
The stronger theorem asserts that
$$
 \lambda\longmapsto\sqrt{\lambda+d_n}\,z_{n,j}(\lambda)
$$
is increasing.

Nearly two decades later, Gautschi returned to this history in
\cite{Gautschi2018}. His high-precision divided-difference experiments
suggested that the known first-derivative inequalities are only the
first members of an infinite sequence of alternating inequalities.
Recall that a $C^\infty$-function $q$ on an open interval is
completely monotone if
$$
 (-1)^m q^{(m)}(\lambda)\geq0,\quad m=0,1,2,\ldots.
$$
The three positive conjectural assertions relevant here are the
following. The first and third can be stated exactly as printed.
Gautschi's second formulation uses $\lambda\geq0$ and describes
the first derivative as completely monotone on $[0,\infty)$.
Section~\ref{sec:h} will show that this derivative tends to
$+\infty$ as $\lambda\downarrow0$, so it is not an ordinary
finite derivative at the endpoint. We therefore state the
mathematically well-defined open-interval form, which contains every
conjectured derivative inequality away from that singular endpoint.

\begin{conjecture}[Gautschi, 2018]\label{conj:f}
For every $n\geq3$ and every positive zero $z_{n,j}$,
$$
 (-1)^m\frac{\dd^{m+1}}{\dd\lambda^{m+1}}
 \left(\sqrt{\lambda+1}\,z_{n,j}(\lambda)\right)>0
$$
for $m=0,1,2,\ldots$ and $\lambda>-1/2$.
See \cite[p.~259]{Gautschi2018}.
\end{conjecture}

\begin{conjecture}[Gautschi, 2018, open-interval form]\label{conj:h}
For every $n\geq2$ and every positive zero $z_{n,j}$,
$$
 (-1)^m\frac{\dd^{m+1}}{\dd\lambda^{m+1}}
 \left(\sqrt{\lambda}\,z_{n,j}(\lambda)\right)>0
$$
for $m=0,1,2,\ldots$ and $\lambda>0$.
This is the restriction to the open interval of the formulation
printed in \cite[p.~263]{Gautschi2018}.
\end{conjecture}

\begin{conjecture}[Gautschi, 2018]\label{conj:g}
Let $n\geq4$. The derivative of
$$
 \sqrt{\lambda+d_n}\,z_{n,j}(\lambda)
$$
is completely monotone for the largest positive zero $j=1$, but is
not completely monotone for any $j\geq2$.
See \cite[p.~263]{Gautschi2018}.
\end{conjecture}

Gautschi also conjectured a more detailed sign pattern for the
derivatives of $\lambda z_{n,j}(\lambda)$. In the case $n=4$, the
part of that conjecture concerning the largest positive zero asserts
$$
 (-1)^m\frac{\dd^{m+1}}{\dd\lambda^{m+1}}
 \bigl(\lambda z_{4,1}(\lambda)\bigr)>0,
 \quad 0\leq m\leq3,\quad \lambda>-1/2.
$$
This is the degree-$4$ portion of the pattern on
\cite[p.~264]{Gautschi2018}. We shall show by an exact calculation
that the assertion already fails for $m=3$.

The purpose of this paper is to prove the first and third
square-root assertions as printed, to establish the open-interval
content of the second, and to disprove the fourth conjecture just
described. The positive statements follow from a stronger
complex-analytic property. A nonnegative function on
$(0,\infty)$ is a complete Bernstein function if it admits the
standard representation
$$
 f(s)=a+bs+\int_{(0,\infty)}\frac{s}{s+t}\,\nu(\dd t),
 \quad a,b\geq0,
$$
with the usual integrability condition on $\nu$. Equivalently, it
extends holomorphically to the slit plane and preserves the upper
half-plane; see
\cite[Theorem~6.2, pp.~69--75, and Theorem~6.9, pp.~78--79]{Schilling}.
We denote the class of complete Bernstein functions by
$$
\CBF:=\{f:(0,\infty)\to[0,\infty): f
 \text{ is a complete Bernstein function}\}.
$$
Here the integrability condition is
$$
 \int_{(0,\infty)}\frac{1}{1+t}\,\nu(\dd t)<\infty.
$$

Our main result is as follows.

\begin{theorem}\label{thm:main}
Let $z_{n,j}$ be the $j$-th positive zero of $C_n^\lambda$.
\begin{enumerate}[label=\textup{(\arabic*)}]
\item For every $n\geq3$ and
 $1\leq j\leq\lfloor n/2\rfloor$, the translated function
$$
 s\longmapsto
 \sqrt{s+\frac12}\,z_{n,j}\left(s-\frac12\right)
$$
belongs to $\CBF$.
\item For every $n\geq2$ and
$1\leq j\leq\lfloor n/2\rfloor$,
$$
 s\longmapsto\sqrt{s}\,z_{n,j}(s)
$$
belongs to $\CBF$.
\item For every $n\geq4$,
$$
 s\longmapsto
 \sqrt{s+d_n-\frac12}\,
 z_{n,1}\left(s-\frac12\right)
$$
belongs to $\CBF$.
\item If $n\geq4$ and $2\leq j\leq\lfloor n/2\rfloor$, then
the derivative of
$\sqrt{\lambda+d_n}\,z_{n,j}(\lambda)$ is not completely
monotone on $(-1/2,\infty)$.
\item The conjectured derivative pattern for
$\lambda z_{n,j}(\lambda)$ is false. In fact,
if
$$
 K(\lambda)=\lambda z_{4,1}(\lambda),
$$
then $K$ extends analytically across $\lambda=-1/2$ and
$$
 K^{(4)}\left(-\frac12\right)=\frac{11}{32}>0.
$$
Consequently $K^{(4)}(\lambda)>0$ at interior points
$\lambda>-1/2$ sufficiently close to $-1/2$, whereas the
conjecture requires this quantity to be negative.
\end{enumerate}
In each of \textup{(1)}--\textup{(3)} the representing measure is
nonzero. Consequently every complete-monotonicity inequality for the
derivative is strict.
\end{theorem}

The key ingredient is a hyperbolicity statement in the
$\lambda$-variable:
$$
 \frac{(-\ii)^n}{\lambda}C_n^\lambda(\ii y)
\quad\hbox{has only real zeros in $\lambda$, for every $y>0$}.
$$
It implies that every positive zero branch maps the upper
$\lambda$-half-plane into the right $x$-half-plane. This gives the
boundary orientation needed for the Pick argument after multiplication
by the appropriate square root.

The largest zero in the degree-dependent scaling requires one additional
oriented-trajectory argument. At the real discriminant point
$\lambda_k=1/2-k$, $k$ zero branches meet at $x=1$. Their local
directions rotate through $\pi/k$ when the parameter passes above
the branch point. The only local sheet that could rotate out of the
upper half-plane is the $k$-th zero branch, rather than the largest
branch. This determines the required boundary sign.

For the remaining zeros the same branch points have the opposite
effect. The $k$-th branch has a first nonremovable singularity
$$
 z_{n,k}(\lambda)
 =1-\kappa_{n,k}^{1/k}
   (\lambda-\lambda_k)^{1/k}
 +O\bigl((\lambda-\lambda_k)^{2/k}\bigr),
$$
with $\kappa_{n,k}>0$. Singularity analysis then forces the eventual
derivative signs to be the strict opposite of complete monotonicity.

The paper is organised as follows. Section~\ref{sec:algebra} records
the discriminant and the collision factorisation. Section~\ref{sec:hyp}
proves parameter hyperbolicity. The right-half-plane property and a
boundary minimum principle are established in
Sections~\ref{sec:geometry} and~\ref{sec:boundary}. The first and
second square-root scalings are treated in Sections~\ref{sec:f}
and~\ref{sec:h}. The oriented largest-zero argument and the
degree-dependent scaling occupy
Sections~\ref{sec:orientation} and~\ref{sec:gpositive}.
Section~\ref{sec:negative} proves the negative assertion for the
remaining branches, and Section~\ref{sec:linear-counterexample}
gives the counterexample to the fourth conjecture.
Appendix~\ref{app:extensions} develops general square-root scales
and further consequences of the same framework; the declaration
following the appendix identifies the provenance of those developments.

\section{Analytic preliminaries}
\label{sec:tools}

We record the precise analytic results used in the sequel and fix the
conventions concerning continuation of algebraic zero branches,
complete Bernstein functions, and singularity analysis.

\subsection{Holomorphic zero branches}

Suppose that $P(\lambda,x)$ is a polynomial in $x$, with
coefficients holomorphic in $\lambda$, and that
$$
 P(\lambda_0,x_0)=0,\quad
 \frac{\partial P}{\partial x}(\lambda_0,x_0)\neq0.
$$
The holomorphic implicit-function theorem provides neighbourhoods
$U$ of $\lambda_0$ and $V$ of $x_0$, and a unique holomorphic
function $z:U\to V$, such that
$$
 z(\lambda_0)=x_0,\quad P(\lambda,z(\lambda))=0.
$$
See \cite[Theorem~B.4, Appendix~B.5, p.~753]{Flajolet}.
Thus a simple polynomial zero moves holomorphically with the
parameter. A branch can fail to continue in this elementary manner
only when it collides with another zero or escapes to infinity because
the degree drops.

We now state exactly what is meant by the discriminant notation used
below. If
$$
 p(x)=a_N\prod_{\ell=1}^{N}(x-\xi_\ell),\quad a_N\ne0,
$$
then its discriminant with respect to the variable $x$ is
$$
\Disc_x p
 :=a_N^{\,2N-2}\prod_{1\leq r<s\leq N}(\xi_r-\xi_s)^2.
$$
The subscript $x$ is not an additional operation: it records that
the roots are taken in the $x$-variable while $\lambda$ is held
fixed. Equivalently,
$$
 \Disc_xp=(-1)^{N(N-1)/2}a_N^{-1}
 \operatorname{Res}_x(p,\partial_xp).
$$
Consequently $\Disc_xp=0$ if and only if $p$ and
$\partial_xp$ have a common zero, that is, if and only if $p$
has a multiple finite root.

There is a separate phenomenon when the leading coefficient
$a_N(\lambda)$ vanishes. The polynomial then loses degree and some
roots may tend to infinity. We call such a parameter a degree-loss
point. The coefficient discriminant remains defined there, but its
vanishing does not by itself distinguish a multiple finite root from a
loss of degree. After monic normalisation, a genuine degree-loss point
appears as a pole of at least one coefficient, once common scalar
factors have been cancelled. We shall therefore inspect both the zeros
of the monic discriminant and the poles of the monic family.

On a parameter domain on which the degree is constant and the
discriminant is nonzero, every root is simple. We also need to know
that local continuation cannot cease along a compact subarc because a
root escapes to infinity. After monic normalisation, write
$$
 P(\lambda,x)=x^N+a_{N-1}(\lambda)x^{N-1}+\cdots+a_0(\lambda).
$$
The coefficients are bounded on every compact parameter set $K$.
The elementary Cauchy bound then gives, uniformly for
$\lambda\in K$,
$$
 |x|\leq1+\max_{0\leq r<N}|a_r(\lambda)|
$$
for every zero $x$ of $P(\lambda,\,\cdot\,)$. Thus the local
implicit branches continue along every path in the domain. Equivalently,
the zero set projects there as a finite unramified covering.
Continuations along two homotopic paths agree by uniqueness of local
roots. If the domain is simply connected, the monodromy theorem then
makes every continued branch single-valued; see
\cite[Chapter~8, Section~1.6, Theorem~2, pp.~295--296]{Ahlfors}. These observations justify the
discriminant calculation in Section~\ref{sec:algebra}.

\subsection{Pick and complete Bernstein functions}

We denote the upper half-plane by
$$
 \Hh=\{z\in\C:\Ima z>0\}.
$$
A holomorphic function $f:\Hh\to\overline{\Hh}$ is called a Pick,
Nevanlinna, or Herglotz function. We shall use the following
characterisation.

\begin{theorem}[Pick--Bernstein characterisation]\label{thm:Pick-CBF}
Let $f:(0,\infty)\to[0,\infty)$. The following are equivalent.
\begin{enumerate}[label=\textup{(\roman*)}]
\item $f$ is a complete Bernstein function.
\item $f$ has a holomorphic extension to
 $\C\setminus(-\infty,0]$ which maps $\Hh$ into
 $\overline{\Hh}$.
\item There are $a,b\geq0$ and a positive measure $\nu$ such that
$$
 f(s)=a+bs+\int_{(0,\infty)}\frac{s}{s+t}\,\nu(\dd t),
$$
where
$$
 \int_{(0,\infty)}(1+t)^{-1}\,\nu(\dd t)<\infty.
$$
\end{enumerate}
\end{theorem}

This is
\cite[Theorem~6.2, pp.~69--75, and Theorem~6.9, pp.~78--79]{Schilling}.
The formulation of Theorem~6.2 in that reference also records that
the right limit $f(0+)$ exists and is real. In the present
formulation this is automatic: Theorem~6.9 identifies the complete
Bernstein functions directly with the Pick functions that are
nonnegative on the positive half-axis. Equivalently, such a function
is nondecreasing there and hence
$$
 0\leq f(0+)=\inf_{s>0}f(s)\leq f(1)<\infty.
$$
Differentiation under the integral is justified on every compact
subinterval of $(0,\infty)$. Indeed, if $s\geq a>0$, then, for
each $m=0,1,2,\ldots$, splitting into $0<t\leq1$ and $t\geq1$
gives a constant $C_{a,m}$ such that
$$
 \frac{t}{(s+t)^{m+2}}
 \leq \frac{C_{a,m}}{1+t},\quad t>0.
$$
The integrability assumption on $\nu$ therefore supplies a common
majorant, so dominated convergence gives
$$
 f'(s)=b+\int_{(0,\infty)}\frac{t}{(s+t)^2}\,\nu(\dd t)
$$
and, for $m\geq1$,
$$
 f^{(m+1)}(s)
 =(-1)^m(m+1)!\int_{(0,\infty)}
 \frac{t}{(s+t)^{m+2}}\,\nu(\dd t).
$$
Therefore
$$
 (-1)^m f^{(m+1)}(s)\geq0,\quad m=0,1,2,\ldots.
$$
If the representing measure is nonzero, the inequalities for
$m\geq1$ are strict; the first one is strict unless both the measure
and the linear coefficient vanish. In our applications a visible
square-root singularity proves that the measure is nonzero.

\subsection{The boundary minimum principle}

The imaginary part of a holomorphic function is harmonic. The
minimum principle for harmonic functions says that a harmonic
function on a bounded domain cannot have a strictly negative interior
minimum if its boundary values are nonnegative; see
\cite[Chapter~4, Section~6.2, Theorem~21, p.~166]{Ahlfors}. Our domains will be upper
half-discs from which small half-discs around real singularities have
been removed. We shall therefore construct a positive barrier near
every removed singularity and obtain a uniform sign estimate on the
large outer semicircle. Both estimates are proved explicitly in
Lemma~\ref{lem:boundary}. This
avoids appealing informally to the boundary behaviour of a multivalued
algebraic function.

\subsection{Algebraic singularities and coefficient transfer}

If $A(z)=\sum_{r=0}^{\infty}a_rz^r$ near the origin, we write
$$
 [z^r]A(z):=a_r.
$$
Thus the square brackets used below denote coefficient extraction,
not an interval or an integer-part operation.

Near a finite singularity $z_0$, each branch of an algebraic
function has a convergent expansion in integral powers of
$(z-z_0)^{1/q}$, for some integer $q\geq1$, on a sufficiently
small disc with a ray from $z_0$ removed. Only finitely many negative
powers occur, and none occur if the branch is bounded. This is the
convergent form of the Newton--Puiseux theorem;
see \cite[Theorem~VII.7, p.~498]{Flajolet}.

For coefficient transfer, the continuation hypothesis must be stated
explicitly. Suppose that $A(z)=\sum_{r=0}^{\infty}a_rz^r$ is
$\Delta$-analytic at its unique singularity $R>0$ of smallest
modulus, in the sense of
\cite[Definition~VI.1, pp.~389--390]{Flajolet}, and that in that
$\Delta$-domain, for some $B\neq0$,
$$
 A(z)=A_0-B(1-z/R)^\alpha+O((1-z/R)^\beta),
 \quad 0<\alpha<\beta,\quad
 \alpha,\beta\notin\{0,1,2,\ldots\}.
$$
Then the transfer theorem gives the complete coefficient estimate
$$
 a_r
 =-\frac{B\,R^{-r}r^{-\alpha-1}}{\Gamma(-\alpha)}
 +O(R^{-r}r^{-\alpha-2})
 +O(R^{-r}r^{-\beta-1}).
$$
The coefficient formula is the standard scale of
\cite[Theorem~VI.1, p.~381]{Flajolet}; passage from a local
singular expansion to coefficient asymptotics is given by the
big-Oh transfer theorem
\cite[Theorem~VI.3, pp.~390--392]{Flajolet}.
Section~\ref{sec:negative} verifies the uniqueness of the dominant singularity
before applying this formula.

\section{Algebraic structure of the zero branches}
\label{sec:algebra}

For the algebraic arguments below, coefficient extraction from the
generating function gives the explicit finite sum
\begin{equation}\label{eq:gegen-sum}
 C_n^\lambda(x)=
 \sum_{r=0}^{\lfloor n/2\rfloor}
 \frac{(-1)^r(\lambda)_{n-r}(2x)^{n-2r}}
 {r!(n-2r)!}.
\end{equation}
This is also \cite[Eq.~(4.7.31), p.~84]{Szego}. Its leading
coefficient is $2^n(\lambda)_n/n!$. At nonpositive
integer parameters the unnormalised polynomial may have a common
normalisation zero. Such factors are to be removed before discussing
zero branches.

Away from the zeros of its leading coefficient, let
$\widehat C_n^\lambda$ denote the monic Gegenbauer polynomial,
$$
 \widehat C_n^\lambda(x)
 :=\frac{n!}{2^n(\lambda)_n}C_n^\lambda(x).
$$
After cancellation, the right-hand side defines a monic rational
family. In particular, its apparent singularity at $\lambda=0$ is
removable: every coefficient of $C_n^\lambda$ contains the factor
$\lambda$, whereas $(\lambda)_n/\lambda$ does not vanish there.
Its coefficients are rational functions of $\lambda$. Accordingly,
$\Disc_x\widehat C_n^\lambda$ below is a rational function; its
zeros record finite multiple roots and its poles record parameters at
which this monic rational family ceases to be regular.
The following formula follows from the Jacobi discriminant recorded
in \cite[Eq.~(3.4.16), p.~69]{Ismail}; see also
\cite[Eq.~(18.16.19)]{DLMF}.

\begin{proposition}\label{prop:disc}
There exists a constant $D_n>0$, independent of $\lambda$, such
that
\begin{equation}\label{eq:disc}
 \Disc_x\widehat C_n^\lambda
 =D_n
 \frac{\displaystyle
 \prod_{r=1}^{\lfloor n/2\rfloor-1}
 (2\lambda+2r+1)^{2r}}
 {\displaystyle
 \prod_{s=\lceil n/2\rceil}^{n-1}
 (\lambda+s)^{2s-1}}.
\end{equation}
In particular, every finite discriminant point in the
$\lambda$-plane is real.
\end{proposition}

\begin{proof}
The monic Gegenbauer polynomial is the monic Jacobi polynomial with
parameters
$$
 \alpha=\beta=\lambda-\frac12.
$$
We first work with $\lambda>0$, where the degrees and all the
uncancelled normalisations occurring below are regular. Both the resulting
discriminant and the expression in \eqref{eq:disc} are rational
functions of $\lambda$. Once the identity has been proved on this
open real interval, the identity theorem for rational functions
extends it meromorphically to the whole $\lambda$-plane. Thus no
formula valid only in the orthogonality range is being substituted
directly at an exceptional parameter.

We spell out the cancellation because the poles in
\eqref{eq:disc} are important later. The Jacobi discriminant formula
\cite[Eq.~(3.4.16), p.~69]{Ismail} is
$$
 \Disc P_n^{(\alpha,\beta)}
 =2^{-n(n-1)}
 \prod_{j=1}^n
 j^{\,j-2n+2}(j+\alpha)^{j-1}(j+\beta)^{j-1}
 (n+j+\alpha+\beta)^{n-j}.
$$
The leading coefficient of $P_n^{(\alpha,\beta)}$ is
$$
 2^{-n}\frac{\Gamma(2n+\alpha+\beta+1)}
 {n!\,\Gamma(n+\alpha+\beta+1)}.
$$
This is \cite[Eq.~(4.21.6), p.~63]{Szego}, written in gamma
notation. 
For a degree-$n$ polynomial, passage to the monic normalisation
divides the discriminant by the $(2n-2)$-nd power of the leading
coefficient. We now insert
$\alpha=\beta=\lambda-1/2$. The factors depending on $\lambda$
are
$$
 \prod_{j=1}^n
 \left(\lambda+j-\frac12\right)^{2j-2}
 (2\lambda+n+j-1)^{n-j},
$$
together with the reciprocal factors coming from the leading
coefficient. We include the exponent count. Apart from constants
independent of $\lambda$, the first family contributes exponent
$r-1$ to $2\lambda+r$ when $r$ is odd and
$1\leq r\leq2n-1$. The second family contributes exponent
$2n-1-r$ when $n\leq r\leq2n-2$. Finally,
$$
 \frac{\Gamma(2n+2\lambda)}
 {\Gamma(n+2\lambda)}
 =\prod_{r=n}^{2n-1}(2\lambda+r)
$$
occurs in the leading coefficient, so monic normalisation contributes
exponent $-(2n-2)$ for $n\leq r\leq2n-1$.

For odd $r<n$, the resulting exponent is $r-1$. For odd
$r\geq n$, the three contributions cancel. For even $r\geq n$,
the resulting exponent is $1-r$; all remaining exponents vanish.
Writing $r=2q+1$ in the first case and $r=2s$ in the last gives
precisely
$$
 \prod_{q=1}^{\lfloor n/2\rfloor-1}
 (2\lambda+2q+1)^{2q}
$$
in the numerator and
$$
 \prod_{s=\lceil n/2\rceil}^{n-1}
 (\lambda+s)^{2s-1}
$$
in the denominator. Powers of $2$ have been absorbed into the
constant, as have all factors depending only on $n$. To determine
its sign, take any
$\lambda>0$: the monic Gegenbauer polynomial then has $n$
distinct real zeros, so its discriminant is positive. Hence the
remaining constant $D_n$ is positive.
\end{proof}

The zeros of the numerator are
$$
 \lambda_k:=\frac12-k,\quad
 2\leq k\leq\left\lfloor\frac n2\right\rfloor.
$$
At these points one has the following factorisation, which we prove
directly.

\begin{proposition}\label{prop:factor}
If $1\leq k\leq\lfloor n/2\rfloor$, then
\begin{equation}\label{eq:factor}
 C_n^{1/2-k}(x)
 =c_{n,k}(1-x^2)^k C_{n-2k}^{k+1/2}(x),
 \quad c_{n,k}\neq0.
\end{equation}
In particular, $x=1$ and $x=-1$ have multiplicity $k$, and all
zeros of the last factor are simple and belong to $(-1,1)$.
\end{proposition}

\begin{proof}
Let
$$
 Y(x)=(1-x^2)^kC_{n-2k}^{k+1/2}(x).
$$
The Gegenbauer differential equation is
\begin{equation}\label{eq:gegen-ode}
 (1-x^2)y''-(2\lambda+1)xy'
 +n(n+2\lambda)y=0.
\end{equation}
See \cite[Eq.~(4.7.5), p.~80]{Szego}; it also follows by
differentiating the generating function twice and comparing
coefficients.
Substitute $\lambda=1/2-k$ in \eqref{eq:gegen-ode}, insert the
displayed expression for $Y$. To make the cancellation verifiable,
write $u=C_{n-2k}^{k+1/2}$. A direct differentiation gives the
operator identity
\begin{align*}
 &(1-x^2)Y''-(2-2k)xY'+n(n+1-2k)Y\\[7pt]
 &\quad=(1-x^2)^k
 \bigl\{(1-x^2)u''-(2k+2)xu'
 +(n-2k)(n+1)u\bigr\}.
\end{align*}
The expression in braces is zero because it is precisely the
Gegenbauer equation of degree $n-2k$ and parameter $k+1/2$.
Thus $Y$ is a polynomial solution of \eqref{eq:gegen-ode} of
degree $n$.

For completeness, inserting
$y=\sum_{\ell=0}^na_\ell x^\ell$ into
\eqref{eq:gegen-ode} gives
$$
 (\ell+2)(\ell+1)a_{\ell+2}
 +(n-\ell)(n+\ell+2\lambda)a_\ell=0.
$$
At $\lambda=1/2-k$, the equation with $\ell=n-1$ first gives
$$
 2(n-k)a_{n-1}=0,
$$
so $a_{n-1}=0$. For $0\leq\ell\leq n-2$, both factors in
$(n-\ell)(n+\ell+1-2k)$ are nonzero, because $n\geq2k$.
Starting from $a_n$, the recurrence therefore determines all
coefficients of the same parity uniquely; starting from
$a_{n-1}=0$, it forces every coefficient of the opposite parity to
vanish. Hence the degree-$n$ polynomial solution is unique up to a
scalar, proving \eqref{eq:factor}. The scalar is nonzero because both
sides have degree $n$.

The factor $(1-x^2)^k$ gives roots of multiplicity $k$ at
$x=1$ and $x=-1$. Since $k+1/2>-1/2$, the residual
Gegenbauer polynomial is orthogonal and has $n-2k$ simple roots in
$(-1,1)$, by \cite[Theorem~3.3.1, pp.~44--46]{Szego}. This proves
the final assertion.
\end{proof}

\begin{remark}\label{rem:degree-loss}
The common factors in $\lambda$ can be removed explicitly. If
$n=2m$, the coefficient of $x^{2m-2r}$ in
$C_{2m}^\lambda/(\lambda)_m$ is
$$
 \frac{(-1)^r2^{2m-2r}}{r!(2m-2r)!}
 (\lambda+m)_{m-r}.
$$
If $n=2m+1$, the coefficient of $x^{2m+1-2r}$ in
$C_{2m+1}^\lambda/(\lambda)_{m+1}$ is
$$
 \frac{(-1)^r2^{2m+1-2r}}{r!(2m+1-2r)!}
 (\lambda+m+1)_{m-r}.
$$
Thus coefficient losses occur only at real parameters. We now prove
the precise growth estimate for every root that escapes to infinity,
because the boundary minimum principle later depends on it.

Consider first $n=2m$, and write
$$
 R(\lambda,x):=\frac{C_{2m}^{\lambda}(x)}{(\lambda)_m}
 =\sum_{r=0}^{m}a_r(\lambda)x^{2m-2r}.
$$
Let $a=-s$ be a degree-loss point,
$m\leq s\leq2m-1$, and put
$$
 q:=2m-s,\quad \delta:=\lambda-a.
$$
The displayed coefficient formula gives the exact valuations
$$
 \operatorname{ord}_{\delta=0}a_r(a+\delta)=1,
 \quad 0\leq r<q,
$$
and
$$
 a_r(a)\neq0,\quad q\leq r\leq m.
$$
In particular,
$$
 a_0(a+\delta)=a_0'(a)\delta+O(\delta^2),
 \quad a_0'(a)\neq0,\quad a_q(a)\neq0.
$$

Roots at infinity must be studied in the reciprocal coordinate
$w=1/x$. Define
$$
 \widetilde R(\delta,w)
 :=w^{2m}R(a+\delta,w^{-1})
 =\sum_{r=0}^{m}a_r(a+\delta)w^{2r}.
$$
To describe all roots tending to infinity, put
$$
 \delta=v^{2q},\quad w=v\xi.
$$
Every term of $\widetilde R(v^{2q},v\xi)$ is divisible by
$v^{2q}$, and division produces a function holomorphic near
$(v,\xi)=(0,\xi_0)$, whose value at $v=0$ is
$$
 a_0'(a)+a_q(a)\xi^{2q}.
$$
This polynomial has $2q$ distinct nonzero roots. The holomorphic
implicit-function theorem therefore gives $2q$ distinct convergent
branches
$$
 w(v)=v\{\xi_0+O(v)\}.
$$
At $\delta=0$, the reciprocal polynomial has a zero of multiplicity
exactly $2q$ at $w=0$. To verify exhaustion, choose
$\rho>0$ so small that $w=0$ is the only zero of
$\widetilde R(0,w)$ in $|w|\leq\rho$. Then
$$
 M_\rho:=\min_{|w|=\rho}|\widetilde R(0,w)|>0.
$$
Uniform convergence on that circle gives, for all sufficiently small
$|\delta|$,
$$
 |\widetilde R(\delta,w)-\widetilde R(0,w)|<M_\rho,
 \quad |w|=\rho.
$$
Rouché's theorem
\cite[Chapter~4, Section~5.2, p.~153]{Ahlfors} therefore gives
exactly $2q$ zeros of $\widetilde R(\delta,\,\cdot\,)$ in the
disc, counted with multiplicity. The $2q$ distinct branches already
constructed account for all of them. Since there are only finitely
many functions $\xi(v)$ and their values at $v=0$ are nonzero,
they remain bounded away from zero in a common disc. Hence every
divergent root satisfies, uniformly in the argument of $\lambda-a$,
$$
 |x(\lambda)|
 =O\bigl(|\lambda-a|^{-1/(2q)}\bigr).
$$
Since $q\geq1$, this also gives
$|x(\lambda)|=O(|\lambda-a|^{-1/2})$.

For $n=2m+1$, factor out the permanent zero $x=0$ and write
$$
 R_{\mathrm o}(\lambda,x)
 :=\frac{C_{2m+1}^{\lambda}(x)}
 {x(\lambda)_{m+1}}
 =\sum_{r=0}^{m}b_r(\lambda)x^{2m-2r}.
$$
The coefficient formula gives
$$
 b_r(\lambda)
 =\frac{(-1)^r2^{2m+1-2r}}
 {r!(2m+1-2r)!}
 (\lambda+m+1)_{m-r}.
$$
Let $a=-s$, where $m+1\leq s\leq2m$, and put
$$
 q:=2m+1-s,\quad \delta:=\lambda-a.
$$
Then
$$
 \operatorname{ord}_{\delta=0}b_r(a+\delta)=1,
 \quad 0\leq r<q,
$$
whereas
$$
 b_r(a)\ne0,\quad q\leq r\leq m.
$$
In particular,
$$
 b_0(a+\delta)=b_0'(a)\delta+O(\delta^2),
 \quad b_0'(a)b_q(a)\ne0.
$$

Introduce the reciprocal polynomial
$$
 \widetilde R_{\mathrm o}(\delta,w)
 :=w^{2m}R_{\mathrm o}(a+\delta,w^{-1})
 =\sum_{r=0}^{m}b_r(a+\delta)w^{2r}.
$$
At $\delta=0$, it has a zero of multiplicity exactly $2q$ at
$w=0$. Put
$$
 \delta=v^{2q},\quad w=v\xi.
$$
Every term of
$\widetilde R_{\mathrm o}(v^{2q},v\xi)$ is divisible by
$v^{2q}$, and
$$
 Q_{\mathrm o}(v,\xi)
 :=v^{-2q}\widetilde R_{\mathrm o}(v^{2q},v\xi)
$$
extends holomorphically to $v=0$, with
$$
 Q_{\mathrm o}(0,\xi)
 =b_0'(a)+b_q(a)\xi^{2q}.
$$
This polynomial has $2q$ distinct nonzero zeros. The holomorphic
implicit-function theorem therefore gives $2q$ distinct branches
$$
 w(v)=v\{\xi_0+O(v)\}.
$$
For fixed sufficiently small $\delta\ne0$, choose one value of
$v$ satisfying $v^{2q}=\delta$; the preceding formulae then give
$2q$ distinct reciprocal roots.

To prove exhaustion, choose $\rho>0$ so small that $w=0$ is the
only zero of $\widetilde R_{\mathrm o}(0,w)$ in
$|w|\leq\rho$, and set
$$
 M_\rho^{\mathrm o}
 :=\min_{|w|=\rho}|\widetilde R_{\mathrm o}(0,w)|>0.
$$
For all sufficiently small $|\delta|$, uniform convergence gives
$$
 |\widetilde R_{\mathrm o}(\delta,w)
   -\widetilde R_{\mathrm o}(0,w)|
 <M_\rho^{\mathrm o},\quad |w|=\rho.
$$
Rouché's theorem
\cite[Chapter~4, Section~5.2, p.~153]{Ahlfors} gives exactly $2q$
zeros in the disc, counted with multiplicity. Hence the constructed
branches exhaust all roots tending to $w=0$, and every divergent
odd-degree zero satisfies, uniformly in the argument of $\lambda-a$,
$$
 |x(\lambda)|
 =O\bigl(|\lambda-a|^{-1/(2q)}\bigr).
$$
The uniformity follows from the same nonvanishing argument as in
even degree. Since $q\geq1$, the bound
$|x(\lambda)|=O(|\lambda-a|^{-1/2})$ follows here as well.
Thus every divergent Gegenbauer zero has boundary growth of order at
most $1/2$.
\end{remark}

\section{Hyperbolicity in the parameter}
\label{sec:hyp}

For $n\geq1$ and $y>0$, define
\begin{equation}\label{eq:Pi}
 \Pi_{n,y}(\lambda)
 :=\frac{(-\ii)^n}{\lambda}C_n^\lambda(\ii y).
\end{equation}
The quotient is a real polynomial in $\lambda$, since
\eqref{eq:gegen-sum} contains the factor $\lambda$.

\begin{proposition}[Parameter hyperbolicity]\label{thm:parameter-hyp}
For every $n\geq2$ and every $y>0$, all zeros of
$\Pi_{n,y}$ are real.
\end{proposition}

\begin{proof}
Substitution of $x=\ii y$ in \eqref{eq:gegen-sum} removes all
alternating signs:
\begin{equation}\label{eq:positive-sum}
 (-\ii)^nC_n^\lambda(\ii y)
 =\sum_{r=0}^{\lfloor n/2\rfloor}
 \frac{(\lambda)_{n-r}(2y)^{n-2r}}
 {r!(n-2r)!}.
\end{equation}
For $m=0,1,2,\ldots$, put
\begin{align}
 S_m(\lambda)
 &:=\frac{(-\ii)^{2m}C_{2m}^\lambda(\ii y)}
 {(\lambda)_m}
 =\sum_{k=0}^m
 \frac{(\lambda+m)_k(2y)^{2k}}
 {(m-k)!(2k)!},                                      \label{eq:S}\\[7pt]
 T_m(\lambda)
 &:=\frac{(-\ii)^{2m+1}C_{2m+1}^\lambda(\ii y)}
 {(\lambda)_{m+1}}
 =\sum_{k=0}^m
 \frac{(\lambda+m+1)_k(2y)^{2k+1}}
 {(m-k)!(2k+1)!}.                                    \label{eq:T}
\end{align}
At a zero of the Pochhammer symbol occurring in either denominator,
the quotient notation in \eqref{eq:S}--\eqref{eq:T} denotes the
polynomial obtained after cancellation. Equivalently, the finite
sums on the right-hand sides define $S_m$ and $T_m$ at every
exceptional parameter.
It follows that
\begin{equation}\label{eq:Pi-factor}
\begin{aligned}
 \Pi_{2m,y}&=(\lambda+1)_{m-1}S_m,
 \quad m=1,2,\ldots,\\[7pt]
 \Pi_{2m+1,y}&=(\lambda+1)_mT_m,
 \quad m=0,1,2,\ldots.
\end{aligned}
\end{equation}
Because every rising factorial $(u)_k$ has nonnegative
coefficients in $u$,
\begin{equation}\label{eq:ST-positive}
 S_m(\lambda)>0,\quad \lambda\geq-m,\quad
 T_m(\lambda)>0,\quad \lambda\geq-m-1.
\end{equation}

The proof uses two coupled recurrences, which we derive from the
three-term Gegenbauer recurrence
$$
 nC_n^\lambda(x)
 =2(n+\lambda-1)xC_{n-1}^\lambda(x)
 -(n+2\lambda-2)C_{n-2}^\lambda(x).
$$
This is \cite[Eq.~(4.7.17), p.~81]{Szego}.
First take $n=2m$, put $x=\ii y$, multiply by
$(-\ii)^{2m}$, and use
$$
 (-\ii)^{2m}C_{2m}^\lambda(\ii y)=(\lambda)_mS_m,
\quad
 (-\ii)^{2m-1}C_{2m-1}^\lambda(\ii y)
 =(\lambda)_mT_{m-1}.
$$
For the moment take $(\lambda)_m\ne0$. After division by
$2(\lambda)_m$, this gives
\eqref{eq:rec1}.  Taking $n=2m+1$ in the same recurrence and using
$$
 (-\ii)^{2m+1}C_{2m+1}^\lambda(\ii y)
 =(\lambda)_{m+1}T_m,
\quad
 (-\ii)^{2m}C_{2m}^\lambda(\ii y)=(\lambda)_mS_m
$$
gives \eqref{eq:rec2}, after division by $(\lambda)_m$. Both final
identities are polynomial identities in $\lambda$. Since they hold
away from the finite zero set of $(\lambda)_m$, they hold for every
$\lambda$ by polynomial continuation. In particular, all subsequent
evaluations at negative integers are legitimate. For
$m=1,2,\ldots$, the two recurrences are
\begin{align}
 mS_m&=y(\lambda+2m-1)T_{m-1}+S_{m-1},               \label{eq:rec1}\\[7pt]
 (2m+1)(\lambda+m)T_m
 &=2y(\lambda+2m)S_m
 +(2\lambda+2m-1)T_{m-1}.                            \label{eq:rec2}
\end{align}
We prove simultaneously, for $m=1,2,\ldots$, that $S_m$ has $m$
simple zeros, all
less than $-m$, that $T_m$ has $m$ simple zeros, all less than
$-m-1$, and that, in increasing order,
\begin{align}
 t_{m,1}<s_{m,1}<\cdots<t_{m,m}<s_{m,m},
 \quad m=1,2,\ldots,                                  \label{eq:inter1}\\[7pt]
 s_{m,1}<t_{m-1,1}<s_{m,2}<\cdots
 <t_{m-1,m-1}<s_{m,m},
 \quad m=2,3,\ldots.                                  \label{eq:inter2}
\end{align}

For $m=1$,
$$
 S_1=1+2y^2(\lambda+1),\quad
 T_1=2y+\frac43y^3(\lambda+2),
$$
so that
$$
 s_{1,1}=-1-\frac1{2y^2},\quad
 t_{1,1}=-2-\frac3{2y^2}.
$$
Thus $t_{1,1}<s_{1,1}$, with
$s_{1,1}<-1$ and $t_{1,1}<-2$, proving the base case.

Let $m=2,3,\ldots$, and assume the assertions through $m-1$.
Both $S_m$ and $T_m$
have degree $m$, and their leading coefficients are respectively
$$
 \frac{(2y)^{2m}}{(2m)!}>0,\quad
 \frac{(2y)^{2m+1}}{(2m+1)!}>0.
$$
Consequently, each has sign $(-1)^m$ at $-\infty$.
At a zero $t_{m-1,r}$, equation \eqref{eq:rec1} reduces to
$$
 mS_m(t_{m-1,r})=S_{m-1}(t_{m-1,r}).
$$
The induction hypothesis \eqref{eq:inter1} says that
$t_{m-1,r}$ lies between the $(r-1)$-st and $r$-th zeros of
$S_{m-1}$, with the evident interpretation when $r=1$.
Since $S_{m-1}$ has positive leading coefficient, exactly
$m-r$ of its zeros lie to the right of $t_{m-1,r}$.  Hence
\begin{equation}\label{eq:S-sign-table}
\operatorname{sgn}S_m(t_{m-1,r})
 =\operatorname{sgn}S_{m-1}(t_{m-1,r})
 =(-1)^{m-r}.
\end{equation}
For $r=1$ this is opposite to the sign $(-1)^m$ of $S_m$ at
$-\infty$.  Consecutive values in
\eqref{eq:S-sign-table} have opposite signs, and the final value
$S_m(t_{m-1,m-1})$ is negative.  Finally,
\eqref{eq:ST-positive} gives $S_m(-m)>0$.  The intermediate-value
theorem therefore produces one zero in each of the $m$ disjoint
intervals
$$
 \begin{gathered}
 (-\infty,t_{m-1,1}),\\[7pt]
 (t_{m-1,r},t_{m-1,r+1}),\quad 1\leq r\leq m-2,\\[7pt]
 (t_{m-1,m-1},-m).
 \end{gathered}
$$
Since $\deg S_m=m$, these are all its zeros and each is simple.
Their locations prove \eqref{eq:inter2}.

If $s_{m,r}$ is a zero of $S_m$, then \eqref{eq:rec2} gives
\begin{equation}\label{eq:T-at-S}
 T_m(s_{m,r})
 =\frac{2s_{m,r}+2m-1}
 {(2m+1)(s_{m,r}+m)}T_{m-1}(s_{m,r}).
\end{equation}
The quotient is positive: indeed $s_{m,r}<-m$, so both
$2s_{m,r}+2m-1$ and $s_{m,r}+m$ are negative.  By
\eqref{eq:inter2}, precisely $m-r$ zeros of $T_{m-1}$ lie to
the right of $s_{m,r}$.  Since $T_{m-1}$ has positive leading
coefficient,
\begin{equation}\label{eq:T-sign-table}
 \operatorname{sgn}T_m(s_{m,r})
 =\operatorname{sgn}T_{m-1}(s_{m,r})
 =(-1)^{m-r}.
\end{equation}
The first sign is opposite to the sign $(-1)^m$ of $T_m$ at
$-\infty$, and consecutive signs alternate.  Hence $T_m$ has one
zero in $(-\infty,s_{m,1})$, and one in each
$(s_{m,r},s_{m,r+1}),\quad 1\leq r\leq m-1$.  These $m$
distinct zeros exhaust its degree and are simple, proving
\eqref{eq:inter1}.  Moreover, $T_m(\lambda)>0$ for
$\lambda\geq-m-1$ by \eqref{eq:ST-positive}; hence every one of
these zeros is strictly smaller than $-m-1$.  This closes the
induction.

Lastly, \eqref{eq:Pi-factor} expresses $\Pi_{n,y}$ as $S_m$ or
$T_m$ multiplied by a rising factorial.  The zeros supplied by the
rising factorials are the real integers
$-1,\ldots,-(m-1)$ in the even case and
$-1,\ldots,-m$ in the odd case.  All remaining zeros are real by
the preceding argument, which proves the proposition.
\end{proof}

The reality of these zeros also follows from classical discrete
orthogonality. With $c=(1+y^2)^{-1}$, the sums
\eqref{eq:S}--\eqref{eq:T} give
$$
 S_m(\lambda)=\frac{M_m(-\lambda-m;1/2,c)}{m!},\quad
 T_m(\lambda)=\frac{2y\,M_m(-\lambda-m-1;3/2,c)}{m!},
$$
where $M_m$ is the Meixner polynomial in
\cite[Eq.~(18.20.7)]{DLMF}. Indeed, substitute
$(2k)!=4^kk!(1/2)_k$, $(2k+1)!=4^kk!(3/2)_k$, and
$m!/(m-k)!=(-1)^k(-m)_k$ in those sums.
Since $0<c<1$, the weights $(\beta)_r c^r/r!$, $r=0,1,\ldots$,
are positive for $\beta=1/2$ and $\beta=3/2$; see
\cite[Table~18.19.1]{DLMF}. The zeros of $M_m$ are therefore simple
and positive by \cite[Theorem~3.3.1, pp.~44--46]{Szego}.
Thus the zeros of $S_m$ lie below $-m$, and those of $T_m$ below
$-m-1$. Together with \eqref{eq:Pi-factor}, this gives another
proof of Proposition~\ref{thm:parameter-hyp}. The preceding
argument also establishes the two coupled interlacings.

\section{Geometry of the analytically continued zeros}
\label{sec:geometry}

For a branch initially positive on $(-1/2,\infty)$, let
$z_{n,j}(\lambda)$ denote its continuation to $\Hh$.
Proposition~\ref{prop:disc} and Remark~\ref{rem:degree-loss} show that
all collision and degree-loss parameters are real. Hence the
continuation exists throughout $\Hh$, and no two zero branches
collide there.

\begin{lemma}\label{lem:right-half}
For every positive zero branch,
$$
 \Rea z_{n,j}(\lambda)>0,\quad \lambda\in\Hh.
$$
\end{lemma}

\begin{proof}
Suppose first that $\Rea z_{n,j}(\lambda)=0$ and
$z_{n,j}(\lambda)=\ii y$, $y\neq0$. By parity, the cases $y>0$
and $y<0$ are equivalent. Equation \eqref{eq:Pi} would give a
nonreal zero $\lambda$ of $\Pi_{n,|y|}$, contradicting
Proposition~\ref{thm:parameter-hyp}.

It remains to exclude $y=0$. If $n$ is even, then
$$
 C_n^\lambda(0)=
 \frac{(-1)^{n/2}(\lambda)_{n/2}}{(n/2)!},
$$
whose zeros in $\lambda$ are real. If $n$ is odd, $x=0$ is the
permanent central zero branch, and a positive branch could reach it
only in a collision. Such a collision is impossible in $\Hh$.
Thus the real part never vanishes. It is positive near the
orthogonality interval and therefore positive throughout the connected
upper half-plane.
\end{proof}

\begin{lemma}[Local splitting at $x=1$]\label{lem:local-splitting}
Let
$$
 \lambda_k=\frac12-k,\quad
 1\leq k\leq\left\lfloor\frac n2\right\rfloor.
$$
If $\delta=\lambda-\lambda_k$ and $u=x-1$, the $k$ local
zero sheets meeting at $x=1$ satisfy
\begin{equation}\label{eq:local-splitting}
 u^k=(-1)^k\kappa_{n,k}\delta+
 O(\delta u,\delta^2,u^{k+1}),\quad
 \kappa_{n,k}
 =\frac{2^kk!(k-1)!(n-2k)!}{n!}>0.
\end{equation}
More precisely, put $\delta=v^k$. For each of the $k$ distinct
numbers $w_0$ satisfying
$$
 w_0^k=(-1)^k\kappa_{n,k},
$$
there is a unique function $w(v)$, holomorphic near $v=0$, such
that
$$
 w(0)=w_0,\quad u=vw(v)
$$
parametrises one of the local sheets. These $k$ parametrisations
exhaust the sheets meeting at $x=1$, and on every fixed sector in
the punctured $\delta$-plane they give the convergent expansions
$$
 u=w_0\delta^{1/k}+O(\delta^{2/k}).
$$
\end{lemma}

\begin{proof}
The Taylor coefficient of $(x-1)^j$ in $C_n^\lambda(x)$ is
$$
 A_j(\lambda)=
 \frac{2^j(\lambda)_j}{j!}
 \frac{(2\lambda+2j)_{n-j}}{(n-j)!}.
$$
Indeed, iterate the derivative identity
$\dd C_r^\mu/\dd x=2\mu C_{r-1}^{\mu+1}$, recorded in
\cite[Eq.~(4.7.14), p.~81]{Szego}, and then use the value at $x=1$.
At $\lambda=\lambda_k$,
$$
 A_0'(\lambda_k)
 =-\frac{2(2k-1)!(n-2k)!}{n!},\quad
 A_k(\lambda_k)
 =(-1)^k\frac{(2k)!}{2^k(k!)^2},
$$
whereas $A_j(\lambda_k)=0$ for $0\leq j<k$, with
$A_0'(\lambda_k)\neq0$. Hence
$$
 0=A_0'(\lambda_k)\delta+A_k(\lambda_k)u^k
 +O(\delta u,\delta^2,u^{k+1}).
$$
Dividing by $A_k(\lambda_k)$ gives
\eqref{eq:local-splitting}.

It remains to justify that this balance gives all the actual local
sheets. The left-hand side before division is a convergent Taylor
series in $(\delta,u)$. Substitute $\delta=v^k$ and $u=vw$.
Every monomial is divisible by $v^k$, and division gives a function
$Q(v,w)$, holomorphic near $v=0$, with
$$
 Q(0,w)=A_0'(\lambda_k)+A_k(\lambda_k)w^k.
$$
The $k$ zeros of this last polynomial are simple. At each of them,
$$
 \frac{\partial Q}{\partial w}(0,w_0)
 =kA_k(\lambda_k)w_0^{k-1}\neq0.
$$
The holomorphic implicit-function theorem therefore gives one and
only one solution $w(v)=w_0+O(v)$ through each $(0,w_0)$.
To verify exhaustion without assuming an a priori rate for $u$,
choose $\rho>0$ so small that $u=0$ is the only zero of
$C_n^{\lambda_k}(1+u)$ in $|u|\leq\rho$. Its multiplicity is
exactly $k$, by Proposition~\ref{prop:factor}. Hence
$$
 M_\rho:=
 \min_{|u|=\rho}|C_n^{\lambda_k}(1+u)|>0.
$$
Uniform convergence on this circle gives, for all sufficiently small
$|\delta|$,
$$
 |C_n^{\lambda_k+\delta}(1+u)
   -C_n^{\lambda_k}(1+u)|<M_\rho,
 \quad |u|=\rho.
$$
Rouché's theorem
\cite[Chapter~4, Section~5.2, p.~153]{Ahlfors} therefore gives
exactly $k$ zeros in $|u|<\rho$, counted with multiplicity. The
$k$ solutions just constructed are distinct for $v\neq0$ small,
so they account for all of them. This proves the asserted exhaustion
and convergence.
\end{proof}

\begin{lemma}[Identification of the colliding sheets]
\label{lem:branch-labels}
Continue the positive branches $z_{n,1},\ldots,
z_{n,\lfloor n/2\rfloor}$ from $\lambda>-1/2$ through the upper
half-plane. At
$$
 \lambda_k=\frac12-k,\quad
 1\leq k\leq\left\lfloor\frac n2\right\rfloor,
$$
exactly the branches $z_{n,1},\ldots,z_{n,k}$ meet at $x=1$.
Every branch $z_{n,j}$, $j>k$, is regular there and tends to the
$(j-k)$-th positive zero, in decreasing order, of
$C_{n-2k}^{k+1/2}$.
\end{lemma}

\begin{proof}
The assertion requires more than counting the positive zeros in
Proposition~\ref{prop:factor}, because after the first collision some
of the continued sheets are no longer real. We therefore follow the
sheets inductively.

Put $M=\lfloor n/2\rfloor$. For $k=1$, all positive zeros are
real and ordered on $(-1/2,\infty)$. Proposition~\ref{prop:factor}
at $\lambda_1=-1/2$ shows that $x=1$ is a simple zero and that
the residual factor has exactly $M-1$ positive simple zeros in
$(0,1)$. Order and continuity therefore give
$$
 z_{n,1}(\lambda)\longrightarrow1
$$
as $\lambda\downarrow\lambda_1$, whilst the other positive
branches tend, in order, to the positive zeros of the residual
factor.

Assume the assertion through $k-1$, and consider the real interval
$$
 I_k=(\lambda_k,\lambda_{k-1}).
$$
After cancellation of the common normalisation factors, the
discriminant formula \eqref{eq:disc} has neither a zero nor a pole in
$I_k$. Thus every boundary sheet is simple throughout this
interval.

We first follow the branches $z_{n,1},\ldots,z_{n,k-1}$. By the
induction hypothesis they meet at $x=1$ when
$\lambda=\lambda_{k-1}$. Apply
Lemma~\ref{lem:local-splitting} with $k-1$. On the left of that
point the leading equation has no negative real direction: if
$k-1$ is even, all its real directions disappear; if $k-1$ is
odd, its only real direction is $u>0$. Consequently each of these
branches is either nonreal on $I_k$, or is real and initially
larger than $1$. A nonreal branch cannot become real inside
$I_k$, since it would then collide with its conjugate. A real
branch cannot cross $x=1$ there. Indeed, for the reduced
polynomials,
$$
 \frac{C_{2m}^{\lambda}(1)}{(\lambda)_m}
 =\frac{2^{2m}(\lambda+1/2)_m}{(2m)!},\quad
 \frac{C_{2m+1}^{\lambda}(1)}{(\lambda)_{m+1}}
 =\frac{2^{2m+1}(\lambda+1/2)_m}{(2m+1)!},
$$
and neither expression vanishes in $I_k$.

At $\lambda_k$, Proposition~\ref{prop:factor} says that every
finite zero is real, the residual zeros are simple and belong to
$(-1,1)$, and $x=1$ has multiplicity $k$. Since
$\lambda_k$ is a half-integer, the leading coefficient is nonzero
there; hence every branch has a finite limit as
$\lambda\downarrow\lambda_k$. If one of the
nonreal branches just described tended to a simple residual zero,
its conjugate branch would tend to the same real zero, contradicting
simplicity. Here the conjugate branch is not being introduced as an
unlabelled extra sheet. Immediately to the left of
$\lambda_{k-1}$, choose a closed disc about $x=1$ which contains
no residual zero at the collision. Let $P(\lambda,x)$ be the
appropriate reduced polynomial from Remark~\ref{rem:degree-loss}, and
let $\Gamma$ be the boundary circle. Then
$$
 M_\Gamma:=
 \min_{x\in\Gamma}|P(\lambda_{k-1},x)|>0.
$$
For real $\lambda<\lambda_{k-1}$ sufficiently close to that point,
uniform convergence on $\Gamma$ gives
$$
 |P(\lambda,x)-P(\lambda_{k-1},x)|<M_\Gamma,
 \quad x\in\Gamma.
$$
Rouché's theorem
\cite[Chapter~4, Section~5.2, p.~153]{Ahlfors} therefore gives
exactly $k-1$ roots in the disc, counted with multiplicity. This
finite cluster is invariant under conjugation because the parameter
and the polynomial coefficients are real. Uniqueness of continuation
and the absence of collisions on $I_k$ preserve that invariance
after the roots leave the fixed disc. Hence the conjugate of any
nonreal branch in the cluster is a distinct branch in the same
cluster. Thus the two branches tending to the same residual zero
would indeed be distinct branches of the family under consideration.
Lemma~\ref{lem:right-half} also excludes a limit with negative real
part, and the simple central zero in odd degree cannot receive a
conjugate pair. A real branch larger than $1$ cannot tend to a
residual zero without crossing $x=1$. Hence all the first
$k-1$ branches tend to $1$.

The branches $z_{n,k},\ldots,z_{n,M}$ tend at
$\lambda_{k-1}$ to distinct positive zeros of the residual factor.
They therefore continue as simple roots throughout $I_k$. To see
explicitly that they remain real, let $z$ be the local holomorphic
germ of one such root at a real parameter $\lambda_0\in I_k$. Since
the polynomial coefficients are real there,
$$
 \widetilde z(\lambda)
 :=\overline{z(\overline\lambda)}
$$
is another holomorphic zero germ and
$\widetilde z(\lambda_0)=z(\lambda_0)$. Uniqueness in the
holomorphic implicit-function theorem gives $\widetilde z=z$.
Hence the branch is real on the real axis near $\lambda_0$; by
continuation it cannot leave the real axis before a collision. As the
discriminant does not vanish on $I_k$, these branches remain real
throughout that interval, and their order cannot change. None can
cross the origin. In even degree,
a zero at the origin would have even multiplicity by parity; in odd
degree it would collide with the permanent simple central zero.
Either event would contradict the absence of a discriminant zero in
$I_k$. Thus these branches remain positive.

At $\lambda_k$ there remains one place at $x=1$ and exactly
$M-k$ positive residual zeros. The largest of the remaining real
branches, namely $z_{n,k}$, must tend to $1$: if a lower branch
had that limit, preservation of order would force $z_{n,k}$ to
have a limit at least $1$, whereas all the unused residual zeros
are strictly smaller than $1$. Lemma~\ref{lem:right-half} excludes
a negative limiting zero. In even degree, the factorisation at
$\lambda_k$ has no zero at the origin. In odd degree, the origin
is a simple zero there, so a noncentral branch cannot tend to it:
the permanent root and a second root would have the same limit,
contradicting local uniqueness. The other branches must therefore tend, in order, to the
$M-k$ positive residual zeros. This closes the induction.
\end{proof}

The slit plane
$$
 \Omega:=\C\setminus(-\infty,-1/2]
$$
is simply connected. It is important to exclude coefficient poles as
well as discriminant zeros. For $n=2m$, use the reduced family
$C_{2m}^{\lambda}/(\lambda)_m$; for $n=2m+1$, use
$C_{2m+1}^{\lambda}/(\lambda)_{m+1}$. The coefficient formula in
Remark~\ref{rem:degree-loss} shows that both families are polynomial
in $(\lambda,x)$. Their leading coefficients can vanish only at
the negative integers listed there, all of which lie on the removed
ray. Proposition~\ref{prop:disc} likewise places every discriminant
zero on that ray. Hence the monic reduced family is holomorphic, has
constant degree, and has nonzero discriminant throughout $\Omega$.
The implicit-function theorem and the monodromy theorem recalled in
Section~\ref{sec:tools} now give every positive branch a
single-valued holomorphic continuation to $\Omega$.

\section{A boundary minimum principle}
\label{sec:boundary}

We shall use a version of the minimum principle adapted to algebraic
singularities on the boundary.

\begin{lemma}\label{lem:boundary}
Let $H$ be holomorphic in $\Hh$, and let $E\subset\R$ be
finite. Assume that $H$ extends holomorphically through every point
of $\R\setminus E$, and that, for each $a\in E$, the selected
branch is algebraic in a punctured upper half-disc about $a$. Thus
it has there a convergent Puiseux expansion on the chosen upper
branch. Suppose
$$
 \liminf_{\substack{z\to x\\[7pt]z\in\Hh}}\Ima H(z)\geq0,
 \quad x\in\R\setminus E.
$$
Assume that, for each $a\in E$,
$$
 H(z)=O(|z-a|^{-\alpha_a}),\quad \alpha_a<1,
$$
uniformly as $z$ approaches $a$ within that upper half-disc,
and assume that, in a neighbourhood of infinity, $H$ has the
convergent Laurent expansion
$$
 H(z)=h+\frac c z+\sum_{\nu\geq2}\frac{c_\nu}{z^\nu},
 \quad h,c,c_\nu\in\R,\quad c<0,
$$
Then
$$
 \Ima H(z)>0,\quad z\in\Hh.
$$
\end{lemma}

\begin{proof}
We give the exhaustion argument in detail. Write
$$
 E=\{a_1,\ldots,a_N\}.
$$
Choose $R>0$ so large that all $a_j$ lie in $(-R,R)$, and
choose $r>0$ so small that the closed intervals
$[a_j-r,a_j+r]$ are disjoint. Let
$$
 D_{R,r}:=
 \{z\in\Hh:|z|<R,\ |z-a_j|>r,\ 1\leq j\leq N\}.
$$
This is a bounded domain. Its boundary has three types of pieces:

\begin{enumerate}[label=\textup{(\arabic*)}]
\item intervals on the real axis;
\item small upper semicircles $|z-a_j|=r$;
\item the large upper semicircle $|z|=R$.
\end{enumerate}

The real intervals are the easiest pieces. By hypothesis, the lower
limit of $\Ima H$ as one approaches any such interval from above is
nonnegative. Since the selected algebraic branch is regular at every
point of $\mathbb R\setminus E$, it extends holomorphically through
each of these intervals. Its upper boundary value is therefore
continuous there, and the stated lower-limit condition says precisely
that this boundary value is nonnegative.

We next construct a barrier for the small semicircles. For
$a\in E$, set
$$
 b_a(z):=\Ima\frac{-1}{z-a}
 =\frac{\Ima z}{|z-a|^2}>0
$$
in $\Hh$. If $z=a+re^{\ii\theta}$, $0<\theta<\pi$, then
$$
 b_a(z)=\frac{\sin\theta}{r}.
$$
We now derive, rather than assume, the angular estimate needed on
these semicircles. By \cite[Theorem~VII.7, p.~498]{Flajolet},
algebraicity gives a convergent fractional-power expansion on the
chosen upper branch:
$$
 H(a+\zeta)=\sum_{\ell=\ell_0}^{\infty}
 c_\ell\zeta^{\ell/N},\quad
 0<|\zeta|<\rho,\quad 0<\arg\zeta<\pi,
$$
for some integer $N\geq1$.  There are only finitely many negative
exponents.  The growth assumption implies
$\ell_0/N\geq-\alpha_a$. Choose $0<\rho_1<\rho$. For
$0<r\leq\rho_1$, define
$$
 \Phi_r(\theta):=
 r^{\alpha_a}\Ima H(a+re^{\ii\theta}),
 \quad 0\leq\theta\leq\pi.
$$
Termwise differentiation of the uniformly convergent Puiseux series
shows that
$$
 \sup_{0<r\leq\rho_1}\sup_{0\leq\theta\leq\pi}
 |\Phi_r'(\theta)|\leq M_a
$$
after $\rho_1$ has been decreased if necessary. Indeed, the
$\ell$-th differentiated term is bounded by a constant times
$r^{\alpha_a+\ell/N}$, and every exponent occurring here is
nonnegative.

At $\theta=0$ and $\theta=\pi$, the points $a+r$ and $a-r$
belong to regular real intervals when $r$ is small.  The boundary
hypothesis therefore gives
$\Phi_r(0)\geq0$ and $\Phi_r(\pi)\geq0$.  The mean-value theorem
now yields
$$
 \min\{\Phi_r(\theta),0\}
 \geq-M_a\min\{\theta,\pi-\theta\}.
$$
Since
$$
 \min\{\theta,\pi-\theta\}\leq\frac{\pi}{2}\sin\theta,
 \quad0\leq\theta\leq\pi,
$$
we obtain the required uniform estimate
$$
 \min\{\Ima H(a+re^{\ii\theta}),0\}
 \geq-C_ar^{-\alpha_a}\sin\theta.
$$

Since $\alpha_a<1$, the $r^{-1}$ barrier grows faster than the
possible negative part in the preceding estimate, uniformly all the way to the
endpoints.
More precisely, for each fixed $\varepsilon>0$, choose $r$ so
small that
$$
 \varepsilon r^{-1}\geq C_ar^{-\alpha_a},
\quad\text{or equivalently}\quad
 r^{1-\alpha_a}\leq\varepsilon/C_a.
$$
Then
$$
 u_\varepsilon(z):=
 \Ima H(z)+\varepsilon\sum_{a\in E}b_a(z)
$$
is nonnegative on every small semicircle once $r$ is sufficiently
small. The barriers centred at the other points remain bounded there
and do not affect this conclusion.

It remains to control the outer boundary. For
$z=Re^{\ii\theta}$,
$$
 \Ima H(z)
 =\frac{-c\sin\theta}{R}
 -\sum_{\nu\geq2}\frac{c_\nu\sin(\nu\theta)}{R^\nu}.
$$
There is $R_0>0$ such that the Laurent series converges for
$|z|>R_0$. Choose $R_1>R_0$. Standard convergence properties of
power series, applied in the variable $1/z$, show that the
termwise derivative converges absolutely on $|z|=R_1$. Hence
$$
 C:=\sum_{\nu\geq2}\nu|c_\nu|R_1^{\,2-\nu}<\infty.
$$
For every $R\geq R_1$, the elementary inequality
$|\sin(\nu\theta)|\leq\nu\sin\theta$ for
$0\leq\theta\leq\pi$ therefore gives
$$
 \left|
 \sum_{\nu\geq2}\frac{c_\nu\sin(\nu\theta)}{R^\nu}
 \right|
 \leq\frac{C\sin\theta}{R^2}.
$$
Consequently,
$$
 \Ima H(Re^{\ii\theta})
 \geq\sin\theta\left(\frac{-c}{R}-\frac C{R^2}\right).
$$
Because $c<0$, the last expression is nonnegative whenever
$R>C/(-c)$. The additional barrier terms are positive, so
$u_\varepsilon\geq0$ on the large semicircle as well.

We may now apply the harmonic minimum principle
\cite[Chapter~4, Section~6.2, Theorem~21, p.~166]{Ahlfors}
to $u_\varepsilon$ on
$D_{R,r}$. It yields
$$
 u_\varepsilon(z)\geq0,\quad z\in D_{R,r}.
$$
There is no hidden boundary-regularity assumption in this application.
After the preceding holomorphic continuations are used to assign the
upper boundary values on the real intervals, $u_\varepsilon$ is
continuous on the closure of $D_{R,r}$. The small and large
semicircles avoid $E$, and their endpoints lie on those regular
intervals.
Fix an arbitrary $z_0\in\Hh$. First choose $R>|z_0|$ as above,
then let $r\downarrow0$. The point $z_0$ remains in the domains,
and we obtain
$$
 \Ima H(z_0)+\varepsilon
 \sum_{a\in E}\Ima\frac{-1}{z_0-a}\geq0.
$$
Finally let $\varepsilon\downarrow0$. This proves
$\Ima H(z_0)\geq0$. Since $z_0$ was arbitrary,
$\Ima H\geq0$ throughout $\Hh$.

The function $H$ is not constant because its Laurent coefficient
$c$ is nonzero. If $\Ima H$ vanished at an interior point, the
strong minimum principle would force $\Ima H$ to vanish identically.
The Cauchy--Riemann equations would then give $H'=0$, so $H$ would
be constant. This contradiction proves $\Ima H>0$ in $\Hh$.
\end{proof}

\begin{remark}\label{rem:boundary-growth}
The reduced equation is polynomial in $(\lambda,x)$, so each zero
branch is an algebraic function of $\lambda$. For a noncentral
zero, write that equation as $Q(\lambda,x^2)=0$, where $Q$ has
degree $m=\lfloor n/2\rfloor$ in its second variable. Then
$H(\lambda)=\sqrt{\lambda+d}\,x(\lambda)$ satisfies
$$
 (\lambda+d)^m
 Q\left(\lambda,\frac{H(\lambda)^2}{\lambda+d}\right)=0.
$$
After the denominators are cleared, this is a nonzero polynomial
equation in $\lambda$ and $H$, so each scaled branch is algebraic
as well. The local expansion theorem
\cite[Theorem~VII.7, p.~498]{Flajolet} therefore
supplies the convergent local expansions required in
Lemma~\ref{lem:boundary}.
Remark~\ref{rem:degree-loss} implies that every divergent Gegenbauer
zero has boundary growth of order at most $1/2$. At a discriminant
point the colliding branches remain bounded. Each square-root factor
used below is locally bounded at every finite point and vanishes at
its own branch point. Multiplication by such a factor therefore
cannot raise a boundary growth exponent to $1$. Thus the local
hypothesis of Lemma~\ref{lem:boundary} holds in all applications.
\end{remark}

\section{The Ismail--Letessier--Askey scaling}
\label{sec:f}

Set
$$
 F_{n,j}(\lambda)
 :=\sqrt{\lambda+1}\,z_{n,j}(\lambda),
$$
where the square root is positive on $(-1,\infty)$.

\begin{theorem}\label{thm:F-pick}
For every $n\geq3$ and every positive zero branch,
$$
 \Ima F_{n,j}(\lambda)>0,\quad \lambda\in\Hh.
$$
\end{theorem}

\begin{proof}
For $a>-1/2$, the upper boundary value is positive real. In the
interval $-1<a<-1/2$, the required branch stays real: the
discriminant formula in Proposition~\ref{prop:disc} has neither a zero
nor a pole there, so continuation from the right cannot create a
multiple zero or a loss of degree. Nor can a
positive branch cross the origin: for even $n$,
$$
 C_n^a(0)=\frac{(-1)^{n/2}(a)_{n/2}}{(n/2)!}\neq0
 ,\quad -1<a<-1/2,
$$
whereas for odd $n$ the origin is the permanent central zero and
meeting it would be a collision with that branch.  Hence the
continued positive branch remains positive and real throughout
$(-1,-1/2)$.

For $a<-1$, the upper boundary value of the scale is
$\ii\sqrt{|a+1|}$. At a regular point $a$, the algebraic branch
has a finite upper boundary value. Since
$\Rea z_{n,j}(a+\ii y)>0$ for every $y>0$,
continuity as $y\downarrow0$ gives
$\Rea z_{n,j}(a+\ii0)\geq0$. Consequently,
$$
 \Ima F_{n,j}(a+\ii0)
 =\sqrt{|a+1|}\,\Rea z_{n,j}(a+\ii0)\geq0.
$$

We finally check infinity. Define the Hermite polynomial $H_n$ by
$$
 H_n(w)=(-1)^n e^{w^2}\frac{\dd^n}{\dd w^n}e^{-w^2};
$$
direct differentiation shows that it is the polynomial solution of
$H_n''-2wH_n'+2nH_n=0$ with leading coefficient $2^n$.
Rodrigues' formula also fixes the relevant orthogonality without
requiring any further convention about Hermite normalisations. If
$p$ is a polynomial with $\deg p<n$, then $n$ integrations by
parts, whose boundary terms vanish because of the Gaussian factor,
give
$$
 \int_{\R}H_n(w)p(w)e^{-w^2}\,\dd w
 =\int_{\R}p^{(n)}(w)e^{-w^2}\,\dd w=0.
$$
Thus $H_n$ is the degree-$n$ member of an orthogonal polynomial
system for the positive weight $e^{-w^2}$. Its zeros are therefore
real and simple by \cite[Theorem~3.3.1, pp.~44--46]{Szego}.
Before making a perturbation calculation, we verify its analytic
meaning. For a real constant $d$, put $t=(\lambda+d)^{-1}$ and
define
\begin{equation}\label{eq:scaled-polynomial}
 {\mathscr P}_{n,d}(w,t):=
 \sum_{r=0}^{\lfloor n/2\rfloor}
 \frac{(-1)^r(2w)^{n-2r}}{r!(n-2r)!}
 \prod_{q=0}^{n-r-1}\{1+(q-d)t\}.
\end{equation}
For $t\neq0$, the finite sum \eqref{eq:gegen-sum} shows that
$$
 {\mathscr P}_{n,d}(w,t)
 =t^{n/2}C_n^{\,t^{-1}-d}(w\sqrt t),
$$
where the right-hand side is interpreted through the single-valued
polynomial expression \eqref{eq:scaled-polynomial}. In particular,
${\mathscr P}_{n,d}$ is a polynomial in $(w,t)$, and
$$
 {\mathscr P}_{n,d}(w,0)=\frac{H_n(w)}{n!}.
$$
Every zero $h$ of $H_n$ is consequently simple. The holomorphic
implicit-function theorem supplies a unique zero $W_d(t)$,
holomorphic for $|t|$ small, with $W_d(0)=h$.

We identify this local germ with the globally labelled zero branch.
The coefficients of ${\mathscr P}_{n,d}(\,\cdot\,,t)$ converge to
those of $H_n/n!$, so the convergence is uniform on compact
$w$-sets. Choose disjoint closed discs $D_\ell$, symmetric about
the real axis, one around each Hermite zero $h_\ell$, and with no
Hermite zero on any boundary. Put
$$
 M_\ell:=\min_{w\in\partial D_\ell}
 \left|\frac{H_n(w)}{n!}\right|>0.
$$
Since there are only finitely many discs, uniform convergence gives
one $\tau>0$ such that, whenever $|t|<\tau$,
$$
 \left|{\mathscr P}_{n,d}(w,t)-\frac{H_n(w)}{n!}\right|
 <M_\ell,\quad w\in\partial D_\ell,
$$
simultaneously for every $\ell$. Rouché's theorem
\cite[Chapter~4, Section~5.2, p.~153]{Ahlfors} then shows that each
disc contains exactly one zero of
${\mathscr P}_{n,d}(\,\cdot\,,t)$. These $n$ zeros exhaust its
degree. If $t>0$, the polynomial has real coefficients. The unique
zero in each symmetric disc must therefore be real, since otherwise
its conjugate would be a second zero in the same disc. Their order is
the order of the Hermite zeros and cannot change while they remain
simple. Because multiplication by the positive number
$\sqrt{\lambda+d}$ preserves the order of the real Gegenbauer
zeros for large $\lambda$, the germ through the $j$-th positive
Hermite zero is precisely
$$
 W_d(t)=\sqrt{\lambda+d}\,z_{n,j}(\lambda)
$$
on the positive real axis near $t=0$. In particular, the largest
Gegenbauer zero corresponds to the largest Hermite zero. Thus the
calculation below determines coefficients of a convergent Taylor
series, not merely of a formal expansion.

We also record the analytic step that promotes this real
identification to an identity of complex germs. For each of the
three real values of $d$ used below, choose $\varepsilon>0$ so
small that the change of variables
$$
 \lambda=t^{-1}-d
$$
maps
$$
 \Sigma_\varepsilon
 :=\{t\in\C:0<|t|<\varepsilon\}
   \setminus(-\varepsilon,0]
$$
into the common domain of holomorphy of the selected zero branch and
the square-root scale. On $\Sigma_\varepsilon$, with the branch
positive for $t>0$, both
$$
 W_d(t)
 \quad\hbox{and}\quad
 t^{-1/2}z_{n,j}(t^{-1}-d)
$$
are holomorphic. They agree on the positive real segment by the
preceding paragraph. The identity theorem
\cite[Chapter~4, Section~3.2, p.~127]{Ahlfors} therefore makes them equal
throughout $\Sigma_\varepsilon$. Since $W_d$ is holomorphic at
$t=0$, its convergent Taylor series is consequently the convergent
complex Laurent expansion of the globally labelled scaled branch at
infinity.

Because ${\mathscr P}_{n,d}$ has real coefficients when $d$ is
real, uniqueness also gives
$W_d(\overline t)=\overline{W_d(t)}$. Hence all coefficients of
this Taylor series are real, as required in
Lemma~\ref{lem:boundary}.

For the corresponding positive zero $h$ of $H_n$, perturbation of
the scaled Gegenbauer equation gives
\begin{equation}\label{eq:F-infty}
 F_{n,j}(\lambda)
 =h+\frac{h(5-2n-2h^2)}{8(\lambda+1)}
 +O(\lambda^{-2}).
\end{equation}
For completeness, if
$W_d(t)=\sqrt{\lambda+d}\,z(\lambda)$ and
$t=(\lambda+d)^{-1}$, the scaled differential equation is
$$
 u''-2wu'+2nu+
 t\{-w^2u''+(2d-1)wu'+n(n-2d)u\}=0.
$$
We spell out the first perturbation coefficient.  Write
$$
 u(w,t)=H_n(w)+tQ(w)+O(t^2);
$$
the harmless scalar normalisation has been chosen so that the leading
term is $H_n$.  With
$$
 L_0=\frac{\dd^2}{\dd w^2}-2w\frac{\dd}{\dd w}+2n,
\quad
 L_1=-w^2\frac{\dd^2}{\dd w^2}
 +(2d-1)w\frac{\dd}{\dd w}+n(n-2d),
$$
the coefficient of $t$ is $L_0Q=-L_1H_n$. Taylor's formula
applied to $e^{-(w-t)^2}$ gives
$$
 e^{2wt-t^2}
 =\sum_{r=0}^{\infty}H_r(w)\frac{t^r}{r!}.
$$
Differentiating this identity with respect to $w$ and $t$, and
then comparing coefficients, gives
$$
 H_n'=2nH_{n-1},\quad
 wH_r=\frac12H_{r+1}+rH_{r-1},
 \quad r=1,2,\ldots.
$$
For $n\geq4$, these identities yield
$$
 L_1H_n
 =4n(n-1)(d-n+1)H_{n-2}
 -4n(n-1)(n-2)(n-3)H_{n-4}.
$$
For $n=2$ or $n=3$, the same calculation gives only
$$
 L_1H_n=4n(n-1)(d-n+1)H_{n-2}.
$$
Since
$$
 L_0H_{n-2r}=4rH_{n-2r},
 \quad 0\leq r\leq\left\lfloor\frac n2\right\rfloor,
$$
one may take, modulo an irrelevant multiple of $H_n$, when $n\geq4$,
$$
 Q=n(n-1)(n-1-d)H_{n-2}
 +\frac{n(n-1)(n-2)(n-3)}2H_{n-4}.
$$
When $n=2$ or $n=3$, one instead takes
$$
 Q=n(n-1)(n-1-d)H_{n-2}.
$$
Let $h$ be a zero of $H_n$.  The Hermite recurrence evaluated
successively at $h$ yields
$$
 \frac{H_{n-2}(h)}{H_n'(h)}
 =\frac{h}{2n(n-1)}
$$
and, for $n\geq4$,
$$
 \frac{H_{n-4}(h)}{H_n'(h)}
 =\frac{h(2h^2-2n+3)}
 {4n(n-1)(n-2)(n-3)}.
$$
Write $W_d=h+a_1t+O(t^2)$. The identity
$u(W_d(t),t)=0$ first gives
$$
 a_1=-\frac{Q(h)}{H_n'(h)}.
$$
For $n\geq4$, substitution of the preceding two ratios in $Q$
gives
$$
 a_1=\frac{h\{4d-(2h^2+2n-1)\}}8.
$$
The small degrees require the same substitution in their shorter
formula for $Q$. Since
$$
 H_2(w)=4w^2-2,\quad H_3(w)=8w^3-12w,
$$
every zero $h$ in degree $n=2$ or $n=3$ satisfies
$$
 h(2h^2-2n+3)=0.
$$
Consequently, in those two degrees,
$$
 -\frac{Q(h)}{H_n'(h)}
 =\frac{h(d-n+1)}2
 =\frac{h\{4d-(2h^2+2n-1)\}}8.
$$
Thus the same value of $a_1$ holds in every degree under
consideration.
We have therefore proved, rather than merely quoted, the expansion
$$
 W_d(t)
 =h+\frac{h\{4d-(2h^2+2n-1)\}}8\,t+O(t^2).
$$
Equivalently, the corresponding scaled zero branch satisfies
\begin{equation}\label{eq:Wd-infty}
 \sqrt{\lambda+d}\,z(\lambda)
 =h+
 \frac{h\{4d-(2h^2+2n-1)\}}{8(\lambda+d)}
 +O(\lambda^{-2}),
\end{equation}
of which \eqref{eq:F-infty} is the case $d=1$. The coefficient in
\eqref{eq:F-infty} is strictly negative for $n\geq3$.

The analytic construction in \eqref{eq:scaled-polynomial} shows that
\eqref{eq:F-infty} is the beginning of a convergent real Laurent
expansion. Lemma~\ref{lem:boundary}, together with
Remark~\ref{rem:boundary-growth}, completes the proof.
\end{proof}

\begin{corollary}\label{cor:F-CBF}
The first square-root conjecture holds.
\end{corollary}

\begin{proof}
Put $s=\lambda+1/2$. The function
$$
 f_{n,j}(s):=F_{n,j}(s-1/2)
$$
is holomorphic in $\C\setminus(-\infty,0]$, nonnegative on
$(0,\infty)$, and Pick by Theorem~\ref{thm:F-pick}. Hence
$f_{n,j}\in\CBF$ by the Pick characterisation
\cite[Theorem~6.9, pp.~78--79]{Schilling}. Its derivative is completely
monotone.

The inequalities are strict. Indeed, the expansion
\eqref{eq:F-infty} has a nonzero $1/\lambda$-coefficient, whereas a
complete Bernstein function with zero representing measure is affine.
Since $F_{n,j}$ has a finite limit at infinity, such an affine
function would have to be constant. Hence the representing measure is
nonzero. Differentiating the complete Bernstein representation gives
strict alternating signs at every order.
\end{proof}

\section{The Ismail--Letessier scaling}
\label{sec:h}

The second historical scaling is
$$
 H_{n,j}(\lambda):=\sqrt{\lambda}\,z_{n,j}(\lambda),
 \quad \lambda>0.
$$
Although its proof uses the same boundary principle as the first
scaling, it is not obtained from that result by translating the
parameter: translation changes the zero $z_{n,j}(\lambda)$ as well
as the square-root factor.

\begin{theorem}\label{thm:H-pick}
For every $n\geq2$ and every positive zero branch,
$$
 \Ima H_{n,j}(\lambda)>0,\quad \lambda\in\Hh.
$$
\end{theorem}

\begin{proof}
We use the square root with argument in $(-\pi,\pi)$.  The zero
branch is holomorphic in the upper half-plane, and
Lemma~\ref{lem:right-half} gives
$$
 \Rea z_{n,j}(\lambda)>0,\quad \lambda\in\Hh.
$$
Consequently $H_{n,j}$ is holomorphic there.

We next verify the boundary sign required by
Lemma~\ref{lem:boundary}.  If $a>0$, orthogonality shows that
$z_{n,j}(a)$ is positive and real.  Hence
$$
 \Ima H_{n,j}(a+\ii0)=0.
$$
If $a<0$, the upper boundary value of the principal square root is
$$
 \sqrt{a+\ii0}=\ii\sqrt{|a|}.
$$
At a regular point $a<0$, algebraicity gives a finite upper boundary
value of the zero branch.  Lemma~\ref{lem:right-half} and continuity
as the parameter approaches $a$ vertically imply
$\Rea z_{n,j}(a+\ii0)\geq0$. Hence
$$
 \Ima H_{n,j}(a+\ii0)
 =\sqrt{|a|}\,\Rea z_{n,j}(a+\ii0)\geq0
$$
at every regular boundary point.  The finitely many algebraic
singularities are precisely the exceptional points allowed in
Lemma~\ref{lem:boundary}; their local growth satisfies the hypothesis
by Remark~\ref{rem:boundary-growth}.

It remains to check the sign at infinity.  Apply the expansion
\eqref{eq:Wd-infty} with $d=0$.  If $\rho>0$ is the corresponding
positive zero of the Hermite polynomial $H_n$, defined in
Section~\ref{sec:f}, then
$$
 H_{n,j}(\lambda)
 =\rho-\frac{\rho(2\rho^2+2n-1)}{8\lambda}
 +O(\lambda^{-2}).
$$
The coefficient of $1/\lambda$ is strictly negative because
$\rho>0$ and $n\geq2$.  All the hypotheses of
Lemma~\ref{lem:boundary} are therefore satisfied, and that lemma gives
$\Ima H_{n,j}>0$ throughout $\Hh$.
\end{proof}

\begin{corollary}\label{cor:H-CBF}
The open-interval form of the second square-root conjecture holds,
with strict inequalities at every derivative order.
\end{corollary}

\begin{proof}
Schwarz reflection across the positive real axis continues $H_{n,j}$ from
the upper and lower half-planes to a holomorphic function on
$\C\setminus(-\infty,0]$; see
\cite[Chapter~4, Section~6.5, Theorem~24,
pp.~172--174]{Ahlfors}. On $(0,\infty)$ it is nonnegative, and
Theorem~\ref{thm:H-pick} says that it maps the upper half-plane into
itself.  The Pick characterisation
\cite[Theorem~6.9, pp.~78--79]{Schilling} therefore gives
$$
 H_{n,j}\in\CBF.
$$
In particular, $H_{n,j}'$ is completely monotone.

We give the endpoint asymptotics explicitly, both to prove strictness
and to determine the status of the closed endpoint in the printed
conjecture.
Differentiating the defining generating function at $\lambda=0$
gives
$$
 \lim_{\lambda\to0}\frac{C_n^\lambda(x)}{\lambda}
 =-[t^n]\log(1-2xt+t^2)
 =\frac{2}{n}T_n(x),\quad n\geq1.
$$
The coefficient-extraction notation was fixed in
Section~\ref{sec:tools}. The last identity follows by putting
$x=\cos\theta$, factoring
$1-2xt+t^2=(1-e^{\ii\theta}t)(1-e^{-\ii\theta}t)$, and expanding
the two logarithms.  The zeros of $T_n$ are simple.  The analytic
implicit-function theorem thus shows that each positive zero branch
has a convergent Taylor expansion, beginning with
$$
 z_{n,j}(\lambda)=\zeta_{n,j}+O(\lambda),
 \quad
 \zeta_{n,j}
 =\cos\frac{(2j-1)\pi}{2n}>0.
$$
Therefore
$$
 H_{n,j}(\lambda)
 =\zeta_{n,j}\lambda^{1/2}+O(\lambda^{3/2})
 ,\quad \lambda\downarrow0.
$$
This nonremovable square-root term shows that the measure in the
complete Bernstein representation is nonzero.  Differentiating that
representation proves
$$
 (-1)^mH_{n,j}^{(m+1)}(\lambda)>0,
 \quad m=0,1,2,\ldots,\quad \lambda>0.
$$
Moreover, termwise differentiation of this convergent Puiseux
expansion gives
$$
 (-1)^mH_{n,j}^{(m+1)}(\lambda)
 =\zeta_{n,j}\frac{(2m-1)!!}{2^{m+1}}\,
   \lambda^{-m-1/2}
 +O(\lambda^{-m+1/2}),
$$
where $(-1)!!=1$.  The expression therefore tends to $+\infty$
as $\lambda\downarrow0$. Thus every conjectured sign has a positive
extended limit at the endpoint. This does not produce ordinary
derivatives at $\lambda=0$: the derivatives are not finite there,
so complete monotonicity holds on $(0,\infty)$, not on the closed
interval in the usual finite-valued sense.
\end{proof}

\section{Orientation of the largest-zero branch}
\label{sec:orientation}

Let $Z_n=z_{n,1}$ be the largest positive zero branch.

\begin{lemma}\label{lem:largest-orientation}
For $n\geq4$,
$$
 \Ima Z_n(a+\ii0)\geq0,\quad
 a>-\left\lceil\frac n2\right\rceil.
$$
\end{lemma}

\begin{proof}
The branch is real until the first collision $\lambda_2=-3/2$.
At $\lambda_k=1/2-k$, put
$\epsilon=\lambda-\lambda_k$ and $u=x-1$. By
Lemma~\ref{lem:local-splitting},
$$
 u^k=(-1)^k\kappa_{n,k}\epsilon+
 O(\epsilon u,\epsilon^2,u^{k+1}),\quad
 \kappa_{n,k}>0.
$$
The convergent parametrisations in that lemma make the continuation
unambiguous. Fix the holomorphic germ $w(v)$ supplied there and,
for $\epsilon=re^{\ii\theta}$, take
$v=r^{1/k}e^{\ii\theta/k}$, $0\leq\theta\leq\pi$. If
$w_0=w(0)$ is the incoming leading coefficient, then this same germ
gives
$$
 u=w_0r^{1/k}e^{\ii\theta/k}+O(r^{2/k})
$$
uniformly for $\theta$ in this closed sector. Hence, when
$\epsilon$ makes an upper half-turn from the positive to the
negative real axis, every leading local direction rotates
anticlockwise by exactly $\pi/k$.

It is useful to visualise the first two cases. For $k=2$, the two
local solutions on the right of $\lambda_2$ point in the positive
and negative real directions. Passage above $\lambda_2$ rotates
them by $\pi/2$; the positive solution, which is the largest zero,
moves upwards. For $k=3$, the permitted incoming directions are
$\pi/3$, $\pi$, and $5\pi/3$. The direction $\pi/3$ remains in the
upper half-plane after rotation by $\pi/3$, whereas the negative
real direction $\pi$ would rotate beyond $\pi$. The general case
has the same elementary angular geometry; Figure~\ref{fig:triple-rotation}
records the triple collision without suppressing the exceptional
sheet.

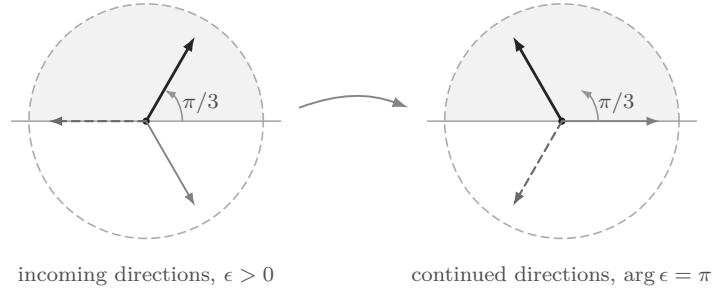
\begin{figure}[!t]
\centering
\begin{tikzpicture}[
  x=1cm,y=1cm,line cap=round,line join=round,
  boundary/.style={
    draw=black!32,
    dash pattern=on 3pt off 2pt,
    line width=.62pt
  },
  upper/.style={fill=black!5},
  axis/.style={draw=black!38,line width=.54pt},
  ordinary/.style={
    -{Latex[length=1.8mm,width=1.25mm]},
    draw=black!50,
    line width=.70pt
  },
  largest/.style={
    -{Latex[length=2.05mm,width=1.4mm]},
    draw=black!86,
    line width=1.05pt
  },
  exceptional/.style={
    -{Latex[length=1.9mm,width=1.3mm]},
    draw=black!58,
    dash pattern=on 3pt off 2pt,
    line width=.82pt
  },
  arcarrow/.style={
    -{Latex[length=1.55mm,width=1.05mm]},
    draw=black!42,
    line width=.62pt
  },
  annotation/.style={font=\scriptsize,text=black!70},
  origin/.style={fill=black!88}
]
\begin{scope}[xshift=-2.75cm]
  \path[upper] (-1.55,0)
    arc[start angle=180,end angle=0,radius=1.55]--cycle;
  \draw[boundary] (0,0) circle (1.55);
  \draw[axis] (-1.78,0)--(1.78,0);
  \fill[origin] (0,0) circle (1.35pt);
  \draw[largest] (0,0)--({1.30*cos(60)},{1.30*sin(60)});
  \draw[exceptional] (0,0)--({1.30*cos(180)},{1.30*sin(180)});
  \draw[ordinary] (0,0)--({1.30*cos(300)},{1.30*sin(300)});
  \draw[arcarrow] (.48,0) arc[start angle=0,end angle=60,radius=.48];
  \node[annotation] at (.72,.25) {$\pi/3$};
  \node[annotation,anchor=north] at (0,-1.78)
    {incoming directions, $\epsilon>0$};
\end{scope}
\begin{scope}[xshift=2.75cm]
  \path[upper] (-1.55,0)
    arc[start angle=180,end angle=0,radius=1.55]--cycle;
  \draw[boundary] (0,0) circle (1.55);
  \draw[axis] (-1.78,0)--(1.78,0);
  \fill[origin] (0,0) circle (1.35pt);
  \draw[largest] (0,0)--({1.30*cos(120)},{1.30*sin(120)});
  \draw[exceptional] (0,0)--({1.30*cos(240)},{1.30*sin(240)});
  \draw[ordinary] (0,0)--({1.30*cos(0)},{1.30*sin(0)});
  \draw[arcarrow] (.48,0) arc[start angle=0,end angle=60,radius=.48];
  \node[annotation] at (.72,.25) {$\pi/3$};
  \node[annotation,anchor=north] at (0,-1.78)
    {continued directions, $\arg\epsilon=\pi$};
\end{scope}
\draw[-{Latex[length=2.1mm,width=1.45mm]},black!48,line width=.72pt]
  (-.72,.18) to[bend left=20] (.72,.18);
\end{tikzpicture}
\caption[Local sheet rotation at the triple collision]
{Local sheet rotation at the triple collision
$\lambda_3=-5/2$. The pale region is the upper $x$-half-plane.
Continuation of the parameter above $\lambda_3$ adds $\pi/3$ to
each incoming direction. The darker ray represents the largest-zero
sheet and remains in the upper half-plane. The dashed ray represents
the unique exceptional incoming sheet; after continuation it lies
below the real axis. The diagram records directions only and is not
drawn to scale.}
\label{fig:triple-rotation}
\end{figure}

Assume inductively that the incoming direction of $Z_n$ belongs to
$[0,\pi]$. The directions allowed by
Lemma~\ref{lem:local-splitting} are
not left implicit.  For $\epsilon>0$ they form the set
$$
 \Theta_k=
 \begin{cases}
 \{\,2\ell\pi/k:0\leq\ell\leq k-1\,\},&k\ \text{even},\\[7pt]
 \{\,(2\ell+1)\pi/k:0\leq\ell\leq k-1\,\},&k\ \text{odd}.
 \end{cases}
$$
In either parity the angles are separated by $2\pi/k$, and
$\pi\in\Theta_k$.  A direction in the closed upper half-plane is
sent below the real axis after addition of $\pi/k$ precisely when
it belongs to $(\pi-\pi/k,\pi]$.  Since the preceding element of
$\Theta_k$ is $\pi-2\pi/k$, one has
$$
 \Theta_k\cap(\pi-\pi/k,\pi]=\{\pi\}.
$$
Thus the negative real direction is the unique exceptional incoming
direction. It corresponds to
$$
 1-\kappa_{n,k}^{1/k}\epsilon^{1/k}
 +O(\epsilon^{2/k}).
$$
This is the $k$-th positive branch, not the largest branch. Indeed,
Lemma~\ref{lem:branch-labels} shows that the $k$-th branch is real
and regular at every earlier collision and reaches $x=1$ for the
first time at $\lambda_k$. It approaches from below and is
therefore precisely the local sheet with negative real incoming
direction.

This also explains why branch labels cannot be exchanged during the
argument. The open upper half-plane is simply connected, and the
discriminant does not vanish there. The monodromy theorem therefore
gives one single-valued continuation of each initial germ throughout
that half-plane. This is the precise reason why the label attached on
$(-1/2,\infty)$ is preserved. We do not appeal to the false general
principle that a permutation of algebraic sheets would require a
collision along the continuation path: such permutations can occur
around branch points on nonsimply connected domains.

For $k=2$, the largest branch is the positive local solution and
rotates into the upper half-plane. The preceding argument closes the
induction.  It remains to justify propagation between consecutive
discriminant points.  On such an interval the upper boundary value is
continuous.  If a non-real boundary zero were to reach the real
axis at an interior parameter, then, because the polynomial has real
coefficients there, it would coincide with its conjugate zero.  The
resulting multiple root would force the discriminant to vanish,
contrary to Proposition~\ref{prop:disc}. If the boundary branch is
real at one regular parameter, the polynomial has real coefficients
there; complex conjugation and uniqueness in the holomorphic
implicit-function theorem then show that its local germ is real on
the real axis. It therefore remains real until a multiple root or a
degree loss is encountered. Thus a branch which is real on an entire
subinterval contributes zero imaginary part and does not violate the
required inequality. Hence the sign cannot change.
The same propagation argument applies from the last collision
$\lambda_{\lfloor n/2\rfloor}$ down to the first pole. The
coefficient formula in Remark~\ref{rem:degree-loss} places that pole
at $-\lceil n/2\rceil$, proving the assertion.
\end{proof}

\section{The Ahmed--Muldoon--Spigler scaling}
\label{sec:gpositive}

Define
$$
 G_n(\lambda):=\sqrt{\lambda+d_n}\,Z_n(\lambda),
 \quad d_n=\frac{2n^2+1}{4n+2}.
$$

\begin{theorem}\label{thm:G-pick}
For every $n\geq4$,
$$
 \Ima G_n(\lambda)>0,\quad \lambda\in\Hh.
$$
\end{theorem}

\begin{proof}
Since
$$
 d_n<\frac n2\leq\left\lceil\frac n2\right\rceil,
$$
Lemma~\ref{lem:largest-orientation} applies throughout $a>-d_n$.
There
$$
 \Ima G_n(a+\ii0)
 =\sqrt{a+d_n}\,\Ima Z_n(a+\ii0)\geq0.
$$
For $a<-d_n$, Lemma~\ref{lem:right-half} gives
$\Rea Z_n(a+\ii0)\geq0$ by the same regular-boundary limiting
argument used in Theorems~\ref{thm:F-pick} and~\ref{thm:H-pick}.
Since the upper boundary value of the square root is
$\ii\sqrt{|a+d_n|}$, we obtain
$$
 \Ima G_n(a+\ii0)
 =\sqrt{|a+d_n|}\,\Rea Z_n(a+\ii0)\geq0.
$$

Let $h_n$ denote the largest positive zero of $H_n$. Formula
\eqref{eq:Wd-infty} yields
\begin{equation}\label{eq:G-infty}
 G_n(\lambda)
 =h_n+
 \frac{h_n\{4d_n-(2h_n^2+2n-1)\}}{8\lambda}
 +O(\lambda^{-2}).
\end{equation}
The coefficient of $w^{n-2}$ in the defining formula for $H_n$
is $-n(n-1)2^{n-2}$, whilst its leading coefficient is $2^n$
and its $w^{n-1}$-coefficient vanishes. Vieta's formula therefore
gives
$$
 \sum_{H_n(\xi)=0}\xi^2=\frac{n(n-1)}2.
$$
All Hermite zeros are real and symmetry pairs each zero with its
negative. Hence every zero has modulus at most $h_n$, and
$$
 h_n^2\geq\frac{n-1}{2}.
$$
Moreover,
$$
 4d_n=2n-1+\frac3{2n+1}<3n-2,\quad n\geq4.
$$
The coefficient of $1/\lambda$ in \eqref{eq:G-infty} is thus
strictly negative. As in the proof of Theorem~\ref{thm:F-pick}, the
implicit function theorem gives a convergent real Laurent expansion.
Lemma~\ref{lem:boundary} now proves the result.
\end{proof}

\begin{corollary}\label{cor:G-CBF}
The positive assertion for the degree-dependent scaling holds, with strict
inequalities at every derivative order.
\end{corollary}

\begin{proof}
Put
$$
 \widetilde G_n(s):=
 \sqrt{s+d_n-\frac12}\,Z_n\left(s-\frac12\right).
$$
It is holomorphic in $\C\setminus(-\infty,0]$, is nonnegative for
$s>0$, and is a Pick function by Theorem~\ref{thm:G-pick}.
Therefore $\widetilde G_n\in\CBF$ by
\cite[Theorem~6.9, pp.~78--79]{Schilling}. The representing measure
is nonzero because \eqref{eq:G-infty} has a nonzero
$1/\lambda$-coefficient. As in the proof of
Corollary~\ref{cor:F-CBF}, a zero representing measure would make
$\widetilde G_n$ affine, and its finite limit at infinity would
then make it constant.
\end{proof}

\section{The remaining zeros and singularity analysis}
\label{sec:negative}

Fix $k\geq2$ and $n\geq2k$. Lemma~\ref{lem:local-splitting}
gives convergent parametrisations of all the local sheets at
$\lambda_k=1/2-k$. We now select the one corresponding to the
$k$-th positive zero.

\begin{lemma}\label{lem:k-Puiseux}
The $k$-th positive branch satisfies
\begin{equation}\label{eq:k-Puiseux}
 z_{n,k}(\lambda)
 =1-\kappa_{n,k}^{1/k}
 (\lambda-\lambda_k)^{1/k}
 +O\bigl((\lambda-\lambda_k)^{2/k}\bigr),
\end{equation}
where $\kappa_{n,k}$ is the positive number in
\eqref{eq:local-splitting}.
\end{lemma}

\begin{proof}
For $\lambda>\lambda_k$, Lemma~\ref{lem:branch-labels} identifies
$z_{n,k}$ as the unique sheet that approaches $1$ from below.
In the convergent parametrisations of
Lemma~\ref{lem:local-splitting}, this is exactly the solution with
$$
 w_0=-\kappa_{n,k}^{1/k}.
$$
Substituting this value in the final expansion of that lemma gives
\eqref{eq:k-Puiseux}.
\end{proof}

\begin{lemma}[The dominant singularity]\label{lem:dominant}
Fix $k\geq2$, $n\geq2k$, and
$\lambda_*>-1/2$. For the branch
$$
 G_{n,k}(\lambda)=
 \sqrt{\lambda+d_n}\,z_{n,k}(\lambda),
$$
the point $\lambda_k=1/2-k$ is the unique nonremovable
singularity at minimum distance from $\lambda_*$. Equivalently,
the function
$$
 t\longmapsto G_{n,k}(\lambda_*-t)
$$
is holomorphic for $|t|<R:=\lambda_*-\lambda_k$, and $t=R$ is
its only singularity on $|t|=R$.
\end{lemma}

\begin{proof}
We examine all possible sources of singularities.

\medskip
\noindent Earlier collision points.
At $\lambda_1=-1/2$, Proposition~\ref{prop:factor} has only simple
endpoint factors, and its residual factor also has simple zeros;
hence this parameter is regular for every zero branch. If
$2\leq j<k$, Lemma~\ref{lem:branch-labels} shows that the
$k$-th branch tends to a simple zero of
$C_{n-2j}^{j+1/2}$. The implicit-function theorem therefore
continues $z_{n,k}$ holomorphically through $\lambda_j$.

\medskip
\noindent Normalisation zeros.
The apparent zero of the unnormalised polynomial at $\lambda=-1$,
and the analogous common factors at other negative integers, do not
make every zero branch singular. They multiply all coefficients by
the same scalar and disappear upon passage to the reduced polynomials
displayed in Remark~\ref{rem:degree-loss}. After this cancellation, a
finite singularity of a zero branch can arise only in one of two
ways: the reduced polynomial has a multiple zero, or its leading
coefficient vanishes and a zero escapes to infinity. The first
possibility is detected by a zero of the monic discriminant
\eqref{eq:disc}; the second is represented by one of its poles and is
also visible directly in Remark~\ref{rem:degree-loss}. Thus the
numerator and denominator of \eqref{eq:disc} provide an exhaustive
finite list; a common normalisation zero is not an additional item.

\medskip
\noindent Nonreal points.
Proposition~\ref{prop:disc} shows that the monic discriminant has no
nonreal zero or pole. Hence no finite nonreal parameter can create a
collision or a loss of degree.

\medskip
\noindent The square-root scale.
The scale has its branch point at $-d_n$. We verify directly that it
lies to the left of $\lambda_k$. Since the smallest value of
$d_n+\lambda_k$ occurs at $k=\lfloor n/2\rfloor$, write either
$n=2m$ or $n=2m+1$. Then
$$
 d_{2m}+\frac12-m=\frac{m+1}{4m+1}>0,
$$
and
$$
 d_{2m+1}+\frac12-m
 =\frac{3(m+1)}{4m+3}>0.
$$
Thus $d_n+\lambda_k>0$, or $-d_n<\lambda_k$.

\medskip
\noindent Degree loss.
The first pole of the monic coefficients is
$-\lceil n/2\rceil$, by the reduced coefficient formula in
Remark~\ref{rem:degree-loss}. Since
$$
 \lambda_k=\frac12-k
 >-\left\lceil\frac n2\right\rceil,\quad
 k\leq\left\lfloor\frac n2\right\rfloor,
$$
every degree-loss point also lies strictly to the left of
$\lambda_k$.

This local list must still be converted into a statement about the
single germ selected at $\lambda_*$. In the disc
$$
 |\lambda-\lambda_*|<\lambda_*-\lambda_k
$$
the only discriminant points are the earlier collisions
$\lambda_j$, $2\leq j<k$. We work with the single reduced
polynomial family of Remark~\ref{rem:degree-loss}; its common
normalisation zeros have already been cancelled and create no
additional punctures.

There is a monodromy issue here which cannot be settled merely by
saying that every preceding point is locally removable. Remove the
finitely many $\lambda_j$ from the disc and join the base point
$\lambda_*$ to a small circle about each $\lambda_j$ by a path
whose interior lies in the upper half-plane. The fundamental group
of the punctured disc
is generated by the resulting lollipop loops. Along the chosen
access path, Lemma~\ref{lem:branch-labels} continues $z_{n,k}$ to
a simple residual zero at $\lambda_j$, because $j<k$. The
implicit-function theorem gives that residual zero a holomorphic
continuation through a full neighbourhood of $\lambda_j$.
Consequently the small circular part of the generator acts trivially
on this sheet. Continuation back along the access path shows that the
complete lollipop loop fixes the original germ at $\lambda_*$.

Every generator fixes the selected germ, and hence so does every
closed path in the punctured disc. The monodromy theorem, in the form
recorded in Section~\ref{sec:tools}, now gives a single-valued
holomorphic branch there. Its local implicit continuations fill in
all the $\lambda_j$, producing one holomorphic function throughout
the full disc. This proves the first assertion without assuming
simultaneous removability.

We have now exhausted all finite singularities relevant to this
branch: the discriminant zeros preceding $\lambda_k$ are removable
for this sheet, the discriminant zero $\lambda_k$ is not, and every
remaining discriminant zero, degree-loss point, and the branch point
of the scale lies to its left.
The point $\lambda_k$ itself is nonremovable by
Lemma~\ref{lem:k-Puiseux}, because $k\geq2$ and hence $1/k$ is
not a nonnegative integer. The factor
$\sqrt{\lambda+d_n}$ is holomorphic and nonzero there, as proved
above, so it cannot remove this fractional-power term.
All other nonremovable points are real and strictly smaller than
$\lambda_k$. If $a<\lambda_k$, then
$$
 |\lambda_*-a|=\lambda_*-a>
 \lambda_*-\lambda_k=R.
$$
Thus they are strictly farther from $\lambda_*$, which proves both
formulations of the lemma.
\end{proof}

\begin{theorem}\label{thm:negative}
Let $k\geq2$ and $n\geq2k$. Then the derivative of
$$
 G_{n,k}(\lambda)
 :=\sqrt{\lambda+d_n}\,z_{n,k}(\lambda)
$$
is not completely monotone on $(-1/2,\infty)$.
\end{theorem}

\begin{proof}
We divide the argument into four steps, including the analytic
hypotheses that are sometimes left implicit in applications of
singularity analysis.

\medskip
\noindent Step 1: the first singular term and its sign.
Lemma~\ref{lem:dominant} includes the inequality
$d_n+\lambda_k>0$. Thus the principal square root is holomorphic
and positive in a neighbourhood of $\lambda_k$. Multiplying
\eqref{eq:k-Puiseux} by its Taylor expansion gives
\begin{equation}\label{eq:Gk-Puiseux}
 G_{n,k}(\lambda)
 =B_0-B_1(\lambda-\lambda_k)^{1/k}
 +O\bigl((\lambda-\lambda_k)^{2/k}\bigr),
 \quad B_1>0.
\end{equation}
More explicitly,
$$
 B_0=\sqrt{d_n+\lambda_k},\quad
 B_1=\sqrt{d_n+\lambda_k}\,\kappa_{n,k}^{1/k}.
$$
The minus sign is therefore intrinsic to the $k$-th zero sheet and
cannot be changed by the scale.

\medskip
\noindent Step 2: the global domain needed for transfer.
Fix an arbitrary $\lambda_*>-1/2$, put
$$
 R=\lambda_*-\lambda_k>0,\quad
 g(t)=G_{n,k}(\lambda_*-t).
$$
Lemma~\ref{lem:dominant} says precisely that $g$ is holomorphic in
$|t|<R$, that $t=R$ is its only singularity on the circle
$|t|=R$, and that every other finite singularity is real and has
modulus strictly larger than $R$. There are only finitely many such
candidate points: they come from the finite zero and pole sets of the
rational monic discriminant, the finite zero set of the reduced
leading coefficient, and the single branch point of the scale.
Consequently their distances from the origin have a positive gap
above $R$. Choose $\eta>0$ smaller than that gap, so that the
closed disc $|t|\leq R+\eta$ contains no other nonremovable
singularity. Remove from this disc the radial segment
$[R,R+\eta]$. Its interior is simply connected, and, after the
earlier collision points described in Lemma~\ref{lem:dominant} have
been filled in, it contains no exceptional parameter of the selected
algebraic sheet. More explicitly, the lollipop generators about those
collision points act trivially by the monodromy calculation in that
lemma; the radial cut removes $t=R$, and the choice of $\eta$
excludes the next nonremovable point. The implicit-function theorem
and the monodromy theorem therefore continue the germ of $g$
uniquely throughout this slit disc. In particular, $g$ is
holomorphic in a standard
$\Delta$-domain
$$
 \{\,|t|<R+\eta,\ t\ne R,\ |\arg(t-R)|>\phi\,\}
$$
for some $0<\phi<\pi/2$. This is a rescaled $\Delta$-domain in
the sense of \cite[Definition~VI.1, pp.~389--390]{Flajolet}. It is
the continuation hypothesis in the transfer theorem; it does not
follow merely from the radius of convergence, which is why it is
recorded separately.

\begin{figure}[!htbp]
\centering
\begin{tikzpicture}[
  x=1cm,y=1cm,line cap=round,line join=round,
  outer/.style={
    draw=black!58,
    fill=black!3,
    line width=.82pt
  },
  taylor/.style={
    draw=black!48,
    fill=black!8,
    line width=.72pt,
    dash pattern=on 3pt off 2pt
  },
  cut/.style={draw=black!66,line width=.76pt},
  axis/.style={
    -{Latex[length=1.9mm,width=1.3mm]},
    draw=black!44,
    line width=.62pt
  },
  point/.style={fill=black!88},
  mainlabel/.style={font=\small,text=black!80},
  annotation/.style={font=\scriptsize,text=black!70}
]
  \def\RR{1.58}
  \def\rho{2.52}
  \def\ang{22}
  \draw[outer] (0,0) circle (\rho);
  \draw[taylor] (0,0) circle (\RR);
  \begin{scope}
    \clip (0,0) circle (\rho);
    \fill[white] (\RR,0)
      --($(\RR,0)+(\ang:3)$)
      --($(\RR,0)+(-\ang:3)$)--cycle;
  \end{scope}
  \draw[cut] (\RR,0)--++(\ang:1.78);
  \draw[cut] (\RR,0)--++(-\ang:1.78);
  \draw[axis] (-2.76,0)--(2.86,0);
  \fill[point] (0,0) circle (1.35pt);
  \node[annotation,anchor=north] at (0,-.08) {$0$};
  \fill[point] (\RR,0) circle (1.55pt);
  \node[annotation,anchor=north] at (\RR,-.08) {$R$};
  \draw[black!38,line width=.5pt] (-1.63,1.76)--(-2.63,2.30);
  \node[mainlabel,anchor=east] at (-2.68,2.30)
    {$\Delta(R,\eta,\phi)$};
  \draw[black!38,line width=.5pt] (2.27,.15)--(3.05,1.38);
  \node[annotation,anchor=west,align=left] at (3.10,1.38)
    {excluded sector};
  \draw[black!38,line width=.5pt] (-1.02,-1.08)--(-2.36,-2.18);
  \node[annotation,anchor=east] at (-2.41,-2.18)
    {Taylor disc, $|t|<R$};
  \draw[black!38,line width=.5pt] (1.72,-1.84)--(2.42,-2.38);
  \node[annotation,anchor=west] at (2.47,-2.38)
    {$|t|=R+\eta$};
  \draw[black!50,line width=.58pt]
    (\RR+.52,0) arc[start angle=0,end angle=\ang,radius=.52];
  \node[annotation] at (\RR+.69,.17) {$\phi$};
\end{tikzpicture}
\caption[$\Delta$-domain used for coefficient transfer]
{The $\Delta$-domain used for coefficient transfer. The dashed
circle bounds the Taylor disc centred at $0$, and $R$ is its
unique boundary singularity. The algebraic branch continues across
that circle into $|t|<R+\eta$, except within a sector of opening
$2\phi$ with vertex $R$. The lightly shaded region represents
$\Delta(R,\eta,\phi)$. The construction is schematic and is not
drawn to scale.}
\label{fig:delta-domain}
\end{figure}
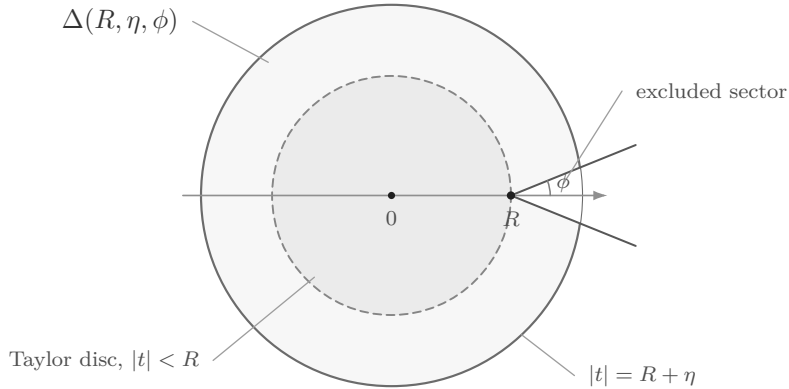

\medskip
\noindent Step 3: separation of integer and fractional powers.
The holomorphic function $w(v)$ constructed in the proof of
Lemma~\ref{lem:k-Puiseux}, together with the ordinary Taylor series
of the square-root scale in $\delta=v^k$, gives a convergent
expansion
$$
 g(t)=\sum_{m=0}^{\infty}c_m(1-t/R)^{m/k}
$$
in that dented neighbourhood, after decreasing it if necessary. Thus
convergence here follows directly from the implicit-function theorem.
From
\eqref{eq:Gk-Puiseux},
$$
 c_1=-B_1R^{1/k}<0.
$$
The terms for which $m/k$ is a nonnegative integer are locally
holomorphic at $R$. To make the remainder estimate explicit,
write $X=1-t/R$. If $k\geq3$, convergence of the Puiseux series
gives, uniformly in a smaller $\Delta$-domain,
$$
 g(t)=c_0+c_1X^{1/k}+O(X^{2/k}).
$$
If $k=2$, retain one further term:
$$
 g(t)=c_0+c_1X^{1/2}+c_2X+O(X^{3/2}).
$$
The terms $c_0$, and $c_2X$ in the second case, are polynomials
in $t$ and hence have no Taylor coefficients of sufficiently high
order. In both cases the exponent in the remainder is strictly
larger than $1/k$. The scale estimate of
\cite[Theorem~VI.1, p.~381]{Flajolet}, together with the remainder
transfer estimate in
\cite[Theorem~VI.3, pp.~390--392]{Flajolet}, therefore yields
$$
 [t^r]g(t)
 \sim-B_1R^{1/k}[t^r](1-t/R)^{1/k}.
$$

\medskip
\noindent Step 4: conversion of the coefficient sign into a
derivative obstruction.
Since $0<1/k<1$,
$$
 [t^r](1-t/R)^{1/k}
 \sim
 \frac{R^{-r}r^{-1-1/k}}{\Gamma(-1/k)}<0.
$$
The sign in the last display is worth making explicit. For
$-1<-1/k<0$, the gamma function is negative; this follows at once
from
$$
 \Gamma(1-1/k)=(-1/k)\Gamma(-1/k)
$$
and the positivity of $\Gamma(1-1/k)$. Multiplication by the
negative coefficient $-B_1$ in \eqref{eq:Gk-Puiseux} therefore makes
the Taylor coefficient positive.
Consequently, for all sufficiently large $r$,
$$
 [t^r]g(t)
 =\frac{(-1)^r}{r!}G_{n,k}^{(r)}(\lambda_*)>0.
$$
Complete monotonicity of $G_{n,k}'$ would instead require
$$
 (-1)^{r-1}G_{n,k}^{(r)}(\lambda_*)\geq0.
$$
The signs are opposite. Thus $G_{n,k}'$ is not completely monotone.
\end{proof}

\section{The linear scaling}
\label{sec:linear-counterexample}

We finish with an exact counterexample to the conjectured sign pattern
for the linear scaling. From the explicit polynomial $C_4^\lambda$,
or simply by solving its quadratic equation in $x^2$, the square of
its largest positive zero is
$$
 z_{4,1}(\lambda)^2
 =
 \frac{3\lambda+6+\sqrt{6\lambda^2+21\lambda+18}}
 {2(\lambda^2+5\lambda+6)}.
$$
The plus sign gives the larger of the two positive zeros because the
denominator is positive for $\lambda>-1/2$. Put
$K(\lambda)=\lambda z_{4,1}(\lambda)$, with the positive square root
in the preceding formula. The radicands and denominators in this
expression are nonzero at $\lambda=-1/2$. Hence the same formula
defines a real-analytic continuation of $K$ to a neighbourhood of
that point.

Put $h=\lambda+1/2$. Substitution in the displayed algebraic formula,
followed by the ordinary binomial expansion of its two square roots,
gives
$$
 K\left(-\frac12+h\right)
 =-\frac12+\frac{13}{12}h-\frac{13}{72}h^2
 +\frac{7}{288}h^3+\frac{11}{768}h^4+O(h^5).
$$
For a direct check which does not involve nested radicals, $K$
satisfies
$$
 4(\lambda+2)(\lambda+3)K^4
 -12\lambda^2(\lambda+2)K^2+3\lambda^4=0.
$$
At $(\lambda,K)=(-1/2,-1/2)$, the derivative of the left-hand side
with respect to $K$ is $-3$. The implicit-function theorem
therefore selects a unique analytic branch there, and substitution of
the proposed series in this polynomial identity verifies its
coefficients successively through $h^4$.
In particular,
$$
 K^{(4)}\left(-\frac12\right)
 =4!\,\frac{11}{768}=\frac{11}{32}>0.
$$
By continuity, there is an $\varepsilon>0$ such that
$$
 K^{(4)}(\lambda)>0,\quad
 -\frac12<\lambda<-\frac12+\varepsilon.
$$
The conjectured inequality for $m=3$ has the opposite sign.
Consequently the conjecture for the linear scaling is false, already
for the largest positive zero in degree $4$.

\begin{proof}[Proof of Theorem~\ref{thm:main}]
Part (1) is Corollary~\ref{cor:F-CBF}; part (2) is
Corollary~\ref{cor:H-CBF}; part (3) is
Corollary~\ref{cor:G-CBF}; and part (4) is
Theorem~\ref{thm:negative}. Part (5) is the exact calculation in
Section~\ref{sec:linear-counterexample}. The strictness assertion in parts
\textup{(1)}--\textup{(3)} is contained in the proofs of the three
corollaries.
\end{proof}

\section{Concluding remarks}\label{sec:conclusion}

The proof separates naturally into three mechanisms. Parameter
hyperbolicity prevents a positive zero branch from crossing the
imaginary $x$-axis when the parameter lies in the upper half-plane.
The real discriminant and the local $k$-fold splitting determine
which boundary sheets lie above the real axis. Finally, the
Pick--Bernstein correspondence converts these geometric facts into
derivative inequalities of every order.

The distinction between the largest and the remaining zeros in the
Ahmed--Muldoon--Spigler scaling is already visible at the first singularity relevant to
each branch. The largest branch has a positive square-root term at
$-3/2$, compatible with a positive Bernstein measure. The $k$-th
branch, $k\geq2$, first becomes singular at
$\lambda_k=1/2-k$ with a negative $k$-th-root term. This sign is
exactly what produces the eventual failure of complete monotonicity.

The argument proves more than the real-variable statements
established here: all three positive scaling functions arise from
complete Bernstein functions, after translation of the parameter
where necessary. Their analytic continuations, Nevanlinna
representations, and sector-preserving properties follow as
additional consequences.

\appendix
\numberwithin{equation}{section}
\section{Further consequences of the zero-branch geometry}
\label{app:extensions}

The results in the main text provide a common framework for these
further consequences. We first record a minimum principle under
weaker hypotheses, then allow the square-root scale to vary
independently of the degree. This gives a sharp threshold for every
positive zero other than the largest, together with a sufficient
range for the largest zero. The final two subsections make explicit
the eventual derivative asymptotics and a convexity consequence of
the complete Bernstein representation.

\subsection{A minimum principle with sublinear boundary growth}
\label{app:minimum}

The argument of Section~\ref{sec:boundary} admits the following
extension. It isolates the boundary growth condition from the
algebraic structure of the zero branches and will also apply to the
additional scalings considered below.

\begin{lemma}\label{lem:appendix-minimum}
Let $H$ be holomorphic in $\Hh$, and let $E\subset\R$ be finite.
Suppose that $H$ extends holomorphically through every point of
$\R\setminus E$ and that
$$
 \Ima H(x+\ii0)\geq0,\quad x\in\R\setminus E.
$$
For each $a\in E$, assume the uniform bound
$$
 H(z)=O(|z-a|^{-\alpha_a}),\quad \alpha_a<1,
 \quad z\to a,\quad z\in\Hh.
$$
Finally, suppose that $H(z)\to h\in\R$ uniformly as
$|z|\to\infty$ within $\Hh$. Then $\Ima H\geq0$ in $\Hh$.
If $H$ is nonconstant, then $\Ima H>0$ there.
\end{lemma}

\begin{proof}
For each $a\in E$, choose
$$
 \max(\alpha_a,0)<\beta_a<1
$$
and, with $0<\arg(z-a)<\pi$, set
$$
 \phi_a(z):=\Rea\bigl(e^{\ii\beta_a\pi/2}
                            (z-a)^{-\beta_a}\bigr).
$$
These functions are harmonic in $\Hh$. For
$z=a+re^{\ii\theta}$, $0\leq\theta\leq\pi$, their continuous
upper boundary values satisfy
$$
 \phi_a(z)
 =r^{-\beta_a}\cos\bigl(\beta_a(\pi/2-\theta)\bigr)
 \geq\cos(\beta_a\pi/2)\,r^{-\beta_a}>0.
$$
In particular, the lower bound is uniform in the angle, including
the endpoints of an upper semicircle.

Fix $z_0\in\Hh$ and $\varepsilon>0$, and consider
$$
 u_\varepsilon(z):=\Ima H(z)
       +\varepsilon\left(1+\sum_{a\in E}\phi_a(z)\right).
$$
Choose $R>|z_0|$ sufficiently large that all points of $E$ lie in
$(-R,R)$ and $\Ima H(z)\geq-\varepsilon$ on the upper
semicircle $|z|=R$. This is possible because the limit at infinity
is real and uniform. For each $a\in E$, choose a sufficiently small
radius $r_a>0$ so that the closed discs about the points of $E$ are
pairwise disjoint, lie inside $|z|<R$, and do not contain $z_0$.
They may further be chosen to satisfy
$$
 \varepsilon\cos(\beta_a\pi/2)\,r_a^{-\beta_a}
       \geq C_a r_a^{-\alpha_a},
$$
where $|H(z)|\leq C_a|z-a|^{-\alpha_a}$ near $a$. Indeed,
$\beta_a>\alpha_a$.

Apply the harmonic minimum principle on the bounded domain
$$
 D:=\{z\in\Hh:|z|<R,\ |z-a|>r_a\text{ for every }a\in E\}.
$$
On its real boundary intervals, $u_\varepsilon\geq0$ by the
boundary hypothesis and the positivity of the barriers. On the
small upper semicircle about $a$, the displayed estimate absorbs
the possible negative part of $\Ima H$. On the outer semicircle,
the constant term $\varepsilon$ absorbs that negative part.
All these boundary values are continuous, so the minimum principle
gives $u_\varepsilon(z_0)\geq0$. Letting $\varepsilon\downarrow0$
at the fixed point $z_0$ proves $\Ima H(z_0)\geq0$.

If equality holds at an interior point, the strong minimum principle
gives $\Ima H\equiv0$. The Cauchy--Riemann equations then force
$H$ to be constant. This proves the strict assertion.
\end{proof}

\begin{remark}\label{rem:appendix-minimum-use}
The lemma requires neither a local Puiseux expansion nor a sign for
the first Laurent coefficient at infinity. For the zero branches
in this paper, Remark~\ref{rem:boundary-growth} supplies the finite
boundary growth estimates. The analytic construction based on
\eqref{eq:scaled-polynomial}, with the identification of the Hermite
zero germ that follows it, supplies the uniform complex limit
$$
 \sqrt{\lambda+d}\,z_{n,j}(\lambda)\longrightarrow h_{n,j}>0,
 \quad |\lambda|\to\infty,\quad\lambda\in\Hh,
$$
where $h_{n,j}$ is the $j$-th positive Hermite zero. That argument
applies to every fixed real $d$: the local polynomial and the
implicit-function construction depend on $d$ only as a fixed
coefficient. Uniformity follows from holomorphy of the selected
germ at $t=(\lambda+d)^{-1}=0$; a limit along the positive real
axis alone would not suffice. Thus the lemma can be used in the
extensions below without imposing a sign on the coefficient in
\eqref{eq:Wd-infty}.
\end{remark}

\subsection{A sharp threshold for the nonlargest zeros}
\label{ap:sec-scales}

The scale parameter can be separated from the degree. For
$d\geq1/2$ and $1\leq k\leq\lfloor n/2\rfloor$, set
\begin{equation}\label{ap:eq-free-scale}
 G_{n,k,d}(\lambda):=\sqrt{\lambda+d}\,z_{n,k}(\lambda),
 \quad
 F_{n,k,d}(s):=G_{n,k,d}(s-1/2),\quad s>0.
\end{equation}
The square root is positive for $\lambda>-1/2$. This restriction
on $d$ makes the scale real throughout the original parameter
interval.

\begin{theorem}[The threshold determined by the zero index]
\label{ap:thm-scale-threshold}
Let $k\geq2$, $n\geq2k$, and $d\geq1/2$. The function
$F_{n,k,d}$ is a complete Bernstein function if and only if
$d\leq k-1/2$. This condition is also necessary and sufficient
for its derivative to be completely monotone on $(0,\infty)$.
When this condition holds,
$$
 (-1)^{r-1}F_{n,k,d}^{(r)}(s)>0,
 \quad r\geq1,\quad s>0.
$$
\end{theorem}

\begin{proof}
Suppose first that $1/2\leq d\leq k-1/2$, and put
$\lambda_k=1/2-k$. The proof of
Lemma~\ref{lem:branch-labels} shows that $z_{n,k}$ is positive
and real on $(\lambda_k,\infty)$, with holomorphic continuation
through the preceding collision points. Thus, at every regular
real boundary point $a>-d$, both factors of $G_{n,k,d}$ are
positive and real. For $a<-d$, Lemma~\ref{lem:right-half} and
continuity of the upper boundary value give
$$
 \Ima G_{n,k,d}(a+\ii0)
 =\sqrt{|a+d|}\,\Rea z_{n,k}(a+\ii0)\geq0.
$$
Only finitely many boundary points are exceptional. The algebraic
continuation and the growth bounds of
Remark~\ref{rem:boundary-growth} apply unchanged: multiplying by
the scale creates at most an additional square-root branch point,
and at that point the scale vanishes. In particular, every local
growth exponent remains less than one. This also covers the
endpoint $d=k-1/2$, where the scale vanishes at the collision.

For completeness, the verification at infinity works for every
fixed real $d$. The polynomial construction
\eqref{eq:scaled-polynomial}, the simple Hermite zeros, and the
identity theorem used in Section~\ref{sec:f} give the convergent
real Laurent expansion
\begin{equation}\label{ap:eq-scale-infinity}
 G_{n,k,d}(\lambda)
 =h_{n,k}+\frac{c_{n,k,d}}{\lambda}+O(\lambda^{-2}),
 \quad
 c_{n,k,d}
 =\frac{h_{n,k}\{4d-(2h_{n,k}^{\,2}+2n-1)\}}8,
\end{equation}
where $h_{n,k}$ is the $k$-th positive Hermite zero in decreasing
order; the coefficient is \eqref{eq:Wd-infty}. Here
$$
 4d\leq4k-2\leq2n-2<2h_{n,k}^{\,2}+2n-1,
$$
so $c_{n,k,d}<0$. Lemma~\ref{lem:boundary}, in its original
form, gives $\Ima G_{n,k,d}>0$ in $\Hh$.
The shifted function is holomorphic in
$\C\setminus(-\infty,0]$ and positive on $(0,\infty)$.
Theorem~\ref{thm:Pick-CBF} therefore gives
$F_{n,k,d}\in\CBF$ and the complete monotonicity of its
derivative. Its finite limit and nonzero Laurent coefficient
exclude an affine function. The representing measure is consequently
nonzero, and the differentiated representation gives all the strict
inequalities, as in Corollary~\ref{cor:F-CBF}.

Conversely, suppose $d>k-1/2$. Lemma~\ref{lem:k-Puiseux}
and the holomorphic, positive scale at $\lambda_k$ yield the
convergent local expansion
\begin{equation}\label{ap:eq-free-puiseux}
 G_{n,k,d}(\lambda)
 =B_0-B_1(\lambda-\lambda_k)^{1/k}
   +O\bigl((\lambda-\lambda_k)^{2/k}\bigr),
\end{equation}
where
$$
 B_0=\sqrt{d+\lambda_k},\quad
 B_1=\sqrt{d+\lambda_k}\,\kappa_{n,k}^{1/k}>0.
$$
Fix $\lambda_*>-1/2$ and set $R=\lambda_*-\lambda_k$.
The continuation argument of Lemma~\ref{lem:dominant} applies
with $d$ in place of $d_n$. Indeed, the only condition on the
scale used there is $-d<\lambda_k$. All degree-loss points
also lie to the left of $\lambda_k$, and each earlier collision
is regular for the selected sheet by
Lemma~\ref{lem:branch-labels}. As in the proof of
Lemma~\ref{lem:dominant}, continuation along an upper access
path reaches a simple residual root at every earlier collision;
the corresponding lollipop loop fixes the germ. These loops
generate the fundamental group of the punctured Taylor disc.
Monodromy and the local implicit continuations therefore give a
single holomorphic branch throughout $|\lambda-\lambda_*|<R$.
The nonzero fractional term in \eqref{ap:eq-free-puiseux} makes
$\lambda_k$ its unique singularity on the boundary circle.

There is now a short Taylor-series obstruction. If $F_{n,k,d}'$
were completely monotone, then, for $0<t<R$,
$$
 G_{n,k,d}'(\lambda_*-t)
 =\sum_{m=0}^{\infty}
 \frac{(-1)^mG_{n,k,d}^{(m+1)}(\lambda_*)}{m!}\,t^m
 \geq0.
$$
Convergence follows from the holomorphic continuation just proved;
coefficient positivity would extend the inequality even to the
part of this interval outside the original parameter domain.
On the other hand, differentiation of the convergent Puiseux
expansion gives
$$
 G_{n,k,d}'(\lambda)
 \sim-\frac{B_1}{k}(\lambda-\lambda_k)^{1/k-1}
 \longrightarrow-\infty,\quad \lambda\downarrow\lambda_k.
$$
This is a contradiction. Since every complete Bernstein function
has a completely monotone derivative, it also excludes
$F_{n,k,d}\in\CBF$ and completes the equivalences.
\end{proof}

\subsection{A larger range of scales for the largest zero}
\label{ap:sec-largest-scale}

The orientation of the largest branch supplies a different
sufficient range, extending beyond the threshold for the other
zeros.

\begin{proposition}\label{ap:prop-largest-scale}
If $n\geq4$ and
$$
 \frac12\leq d\leq\left\lceil\frac n2\right\rceil,
$$
then $F_{n,1,d}\in\CBF$. Moreover,
$(-1)^{r-1}F_{n,1,d}^{(r)}(s)>0$ for every $r\geq1$ and
$s>0$.
\end{proposition}

\begin{proof}
For a regular point $a>-d$,
Lemma~\ref{lem:largest-orientation} gives
$$
 \Ima G_{n,1,d}(a+\ii0)
 =\sqrt{a+d}\,\Ima z_{n,1}(a+\ii0)\geq0.
$$
For $a<-d$, Lemma~\ref{lem:right-half} gives the same sign
after multiplication by $\ii\sqrt{|a+d|}$. The local growth
and holomorphic continuation are as in the preceding proof.

Let $h_n=h_{n,1}$. The coefficient in
\eqref{ap:eq-scale-infinity} is strictly negative throughout the
stated range. Indeed, Section~\ref{sec:gpositive} establishes
$h_n^2\geq(n-1)/2$. For even $n\geq4$, this implies
$h_n^2>1/2$, whence
$$
 4d\leq2n<2h_n^2+2n-1.
$$
For odd $n\geq5$, it implies $h_n^2>3/2$, whence
$$
 4d\leq2n+2<2h_n^2+2n-1.
$$
Lemma~\ref{lem:boundary} and Theorem~\ref{thm:Pick-CBF}
now apply. The nonzero Laurent coefficient and the finite limit
again give a nonzero representing measure and strict inequalities.
\end{proof}

This proposition asserts sufficiency only. No optimality of its
upper endpoint is claimed; the exact classification in
Theorem~\ref{ap:thm-scale-threshold} concerns $k\geq2$.

\begin{remark}[The two lowest degrees]\label{ap:rem-low-degrees}
For $n=2,3$, the largest-zero classification is elementary:
$$
 F_{2,1,d}(s)=\frac1{\sqrt2}
       \sqrt{\frac{s+d-1/2}{s+1/2}},\quad
 F_{3,1,d}(s)=\sqrt{\frac32}
       \sqrt{\frac{s+d-1/2}{s+3/2}}.
$$
Consequently $F_{2,1,d}\in\CBF$ exactly when $1/2\leq d\leq1$,
and $F_{3,1,d}\in\CBF$ exactly when $1/2\leq d\leq2$.
Indeed, $(s+a)/(s+b)$ maps $\Hh$ into itself for
$0\leq a<b$, and so does its principal square root; positivity
on $(0,\infty)$ gives the complete Bernstein property.
For $a>b$ the square root is decreasing on $(0,\infty)$,
which excludes that property. At the upper endpoints the two
functions are constant, so strict derivative inequalities do not
hold there.
\end{remark}

\subsection{The eventual sign at every derivative order}
\label{ap:sec-eventual-sign}

The transfer argument of Section~\ref{sec:negative} gives more
information than the Taylor-series contradiction: at any fixed
interior parameter, every sufficiently high derivative has the
opposite sign to complete monotonicity.

\begin{proposition}\label{ap:prop-eventual-sign}
Let $k\geq2$, $n\geq2k$, $d>k-1/2$, and
$\lambda_*>-1/2$. With
$$
 R=\lambda_*-\lambda_k,\quad
 A_{n,k,d,\lambda_*}
 =\frac{\sqrt{d+\lambda_k}\,\kappa_{n,k}^{1/k}R^{1/k}}
 {|\Gamma(-1/k)|}>0,
$$
one has
\begin{equation}\label{ap:eq-derivative-asymptotic}
 (-1)^{r-1}G_{n,k,d}^{(r)}(\lambda_*)
 \sim-A_{n,k,d,\lambda_*}
       \frac{r!}{R^r r^{1+1/k}},\quad r\longrightarrow\infty.
\end{equation}
In particular, $G_{n,k,d}'$ is not completely monotone on any
nonempty open subinterval of $(-1/2,\infty)$.
\end{proposition}

\begin{proof}
Put $g(t)=G_{n,k,d}(\lambda_*-t)$. The continuation established
in the proof of Theorem~\ref{ap:thm-scale-threshold} makes $R$
the unique singularity on its circle of convergence. After the
earlier collision points have been filled in, all candidate
singularities other than $R$ form a finite real set lying farther
from the origin.
Choose $\eta>0$ so that no further candidate enters
$|t|<R+\eta$. The same monodromy argument, with a radial cut
from $R$, continues $g$ to a $\Delta$-domain
$$
 \{\,|t|<R+\eta,\ t\neq R,\ |\arg(t-R)|>\phi\,\},
 \quad 0<\phi<\pi/2.
$$
Thus the continuation hypothesis for coefficient transfer holds,
in addition to the radius-of-convergence assertion.

Writing $X=1-t/R$, the convergent Puiseux series gives, when
$k\geq3$,
$$
 g(t)=c_0-B_1R^{1/k}X^{1/k}+O(X^{2/k}).
$$
For $k=2$, retain the integer-power term:
$$
 g(t)=c_0-B_1R^{1/2}X^{1/2}+c_2X+O(X^{3/2}).
$$
The displayed integer powers are polynomials in $t$ and have no
coefficients of sufficiently high order. The transfer estimates
used in Section~\ref{sec:negative} therefore imply
$$
 [t^r]g(t)
 \sim-\frac{B_1R^{1/k}}{\Gamma(-1/k)}
           R^{-r}r^{-1-1/k}.
$$
Since $\Gamma(-1/k)<0$ and
$[t^r]g(t)=(-1)^rG_{n,k,d}^{(r)}(\lambda_*)/r!$, this is
\eqref{ap:eq-derivative-asymptotic}. Its right-hand side is
strictly negative. Every nonempty open subinterval contains a
point $\lambda_*$ at which all sufficiently high orders violate
the inequalities required of a completely monotone derivative.
\end{proof}

\subsection{Convexity for every positive zero}
\label{app:convexity}

The complete Bernstein property in Theorem~\ref{thm:main}(2)
also gives a convexity consequence uniform in the zero index.
Dimitrov~\cite[Theorem~1]{Dimitrov2003} proved convexity of
$\lambda^{3/2}z_{n,1}(\lambda)$ on $[0,\infty)$ for the largest
zero. The following deduction extends that conclusion to every
positive zero within the representation framework of this paper.

\begin{corollary}\label{cor:appendix-convexity}
For every $n\geq2$ and
$1\leq j\leq\lfloor n/2\rfloor$, the function
$$
 V_{n,j}(\lambda):=\lambda^{3/2}z_{n,j}(\lambda),
 \quad\lambda>0,
$$
satisfies $V_{n,j}''(\lambda)>0$ on $(0,\infty)$. Its continuous
extension defined by $V_{n,j}(0)=0$ is convex on $[0,\infty)$.
\end{corollary}

\begin{proof}
Let $f\in\CBF$, with the representation in
Theorem~\ref{thm:Pick-CBF}. Differentiating $sf(s)$ twice gives
\begin{equation}\label{eq:appendix-convexity}
 \frac{\dd^2}{\dd s^2}\bigl(sf(s)\bigr)
 =2f'(s)+sf''(s)
 =2b+\int_{(0,\infty)}\frac{2t^2}{(s+t)^3}\,\nu(\dd t).
\end{equation}
The differentiation is justified by the bounds established in
Section~\ref{sec:tools}. Equivalently, on every compact subinterval
of $(0,\infty)$ the kernel on the right is bounded by a constant
multiple of $(1+t)^{-1}$, which is integrable against $\nu$.
The expression is nonnegative, and is strictly positive if
$\nu\neq0$.

Apply this identity with
$$
 f(\lambda)=\sqrt{\lambda}\,z_{n,j}(\lambda).
$$
Theorem~\ref{thm:main}(2) gives $f\in\CBF$, with a nonzero
representing measure, so \eqref{eq:appendix-convexity} yields the
strict inequality. Finally, the reduced Chebyshev limit in
Section~\ref{sec:h} gives
$$
 V_{n,j}(\lambda)
 =\zeta_{n,j}\lambda^{3/2}+O(\lambda^{5/2}),
 \quad
 \zeta_{n,j}=\cos\frac{(2j-1)\pi}{2n}>0,
 \quad\lambda\downarrow0.
$$
Hence the stated extension is continuous at zero. The convexity
inequality with an endpoint at zero follows by letting a positive
endpoint tend to zero in the interior convexity inequality.
\end{proof}


\section*{Declaration on the use of generative AI}

All generalizations and other additional mathematical developments
in Appendix~\ref{app:extensions}, including their arguments and
proofs, were developed and written entirely online by ChatGPT from
the original manuscript, without any human intervention in their
mathematical development, during a video call with Renato
\'Alvarez-Nodarse on 15 September 2026. The model used was OpenAI's
\texttt{gpt-6-astra}.

ChatGPT was not used in any way to generate the mathematical
arguments in Sections~1--\ref{sec:conclusion}, including the proof
of Theorem~\ref{thm:main}. At the author's request, the mathematical
statements and proofs in these sections retain their original form
to distinguish the original contribution clearly from the
generalizations and all other additional mathematical material,
which are collected exclusively in Appendix~\ref{app:extensions}.

ChatGPT independently carried out the entire editorial integration and
typesetting of the Appendix and was explicitly instructed to preserve
the style of the original article. The mathematical content of the
Appendix was reviewed for correctness without altering a single formula.
The entire process of mathematical verification, editorial integration,
and typesetting took approximately 14 minutes, measured from the request
to revise the source file, with the mathematical proof itself having an
estimated prorated cost of approximately \texteuro 0.46 under the
subscription used.

\section*{Acknowledgements}

K. Castillo acknowledges financial support from the Centre for
Mathematics of the University of Coimbra (CMUC), funded by the
Portuguese Foundation for Science and Technology (FCT), under the
projects UID/00324/2025 and UID/PRR/00324/2025. K. Castillo also
acknowledges financial support from the FCT\@. The grant number is
\mbox{2022.00143.CEECIND/CP1714/CT0002}.

\end{document}